\documentclass[a4paper,11pt]{amsart} 
\usepackage{amsmath,amsxtra,amssymb,latexsym, amscd,amsthm}
\usepackage[mathscr]{eucal}
\usepackage{mathrsfs}
\usepackage{bbm}
\usepackage{enumerate}
\usepackage{pict2e}
\usepackage{tikz}
\usepackage{graphicx}
\usepackage[a4paper]{geometry}
\usepackage{color}
\usepackage{hyperref}
\numberwithin{equation}{section}

\theoremstyle{plain}
\newtheorem{thm}{Theorem}[section]

\newtheorem{prp}[thm]{Proposition}
\newtheorem{cor}[thm]{Corollary}
\newtheorem{lem}[thm]{Lemma}
\newtheorem*{euc*}{Euclidean division}
\newtheorem*{fek*}{Fekete's Lemma}
\newtheorem*{kin*}{Kingman's Subadditive Ergodic Theorem}
\newtheorem*{fur*}{Furstenberg-Kesten Theorem}
\theoremstyle{definition}
\newtheorem{defa}[thm]{Definition}

\newtheorem{rem}[thm]{Remark}

\newtheorem*{rem*}{Remark}

\newcommand{\dd}{\mathrm{d}}

\newcommand{\ii}{\mathrm{i}}
\renewcommand{\Im}{\operatorname{Im}}
\renewcommand{\Re}{\operatorname{Re}}

\newcommand{\N}{\mathbb{N}}
\newcommand{\Z}{\mathbb{Z}}

\newcommand{\R}{\mathbb{R}}
\newcommand{\C}{\mathbb{C}}

\newcommand{\T}{\mathbb{T}}

\newcommand{\cT}{\mathcal{T}}
\newcommand{\cN}{\mathcal{N}}
\newcommand{\cA}{\mathcal{A}}

\newcommand{\cW}{\mathcal{W}}

\newcommand{\cB}{\mathcal{B}}
\newcommand{\cM}{\mathcal{M}}
\newcommand{\cK}{\mathcal{K}}

\newcommand{\cZ}{\mathcal{Z}}

\newcommand{\eps}{\epsilon}
\DeclareMathOperator{\dist}{dist}

\DeclareMathOperator{\supp}{supp}

\DeclareMathOperator{\tr}{tr}
\DeclareMathOperator{\rank}{rank}

\DeclareSymbolFont{extraup}{U}{zavm}{m}{n}
\DeclareMathSymbol{\varheart}{\mathalpha}{extraup}{86}
\DeclareMathSymbol{\vardiamond}{\mathalpha}{extraup}{87}

\title{$L^p$ norms and support of eigenfunctions on graphs}
\author{Etienne Le Masson}
\address{Universit\'e de Cergy-Pontoise, AGM, 2 av. Adolphe Chauvin, 95302 Cergy-Pontoise Cedex, France.}
\email{etienne.le-masson@u-cergy.fr}
\author{Mostafa Sabri}
\address{Department of Mathematics, Faculty of Science, Cairo University, Cairo 12613, Egypt.}
\address{Universit\'e Paris Sud XI, UMR 8628 du CNRS, Laboratoire de Math\'ematique, B\^at. 307, 91405 Orsay Cedex, France.}
\email{mmsabri@sci.cu.edu.eg}
\keywords{large graphs, delocalization, $p$-norm of eigenfunctions, support of eigenfunctions, $N$-lifts.}
\usepackage{calc}
\usepackage{graphicx}

\makeatletter
\newlength{\temp@wc@width}
\newlength{\temp@wc@height}
\newcommand{\widecheck}[1]{%
  \setlength{\temp@wc@width}{\widthof{$#1$}}%
  \setlength{\temp@wc@height}{\heightof{$#1$}}%
  #1\hspace{-\temp@wc@width}%
  \raisebox{\temp@wc@height+2pt}[\heightof{$\widehat{#1}$}]%
     {\rotatebox[origin=c]{180}{\vbox to 0pt{\hbox{$\widehat{\hphantom{#1}}$}}}}%
}
\makeatother

\begin{document}

\begin{abstract}
This article is concerned with properties of delocalization for eigenfunctions of Schr\"odinger operators on large finite graphs. More specifically, we show that the eigenfunctions have a large support and we assess their $\ell^p$-norms. Our estimates hold for any fixed, possibly irregular graph, in prescribed energy regions, and also for certain sequences of graphs such as $N$-lifts.
\end{abstract}

\maketitle

\section{Introduction}\label{sec:mainres}

Recent years have seen much interest in understanding the spectra and eigenvectors of large finite graphs and random matrices. Concerning the spectrum, central questions are the convergence of the empirical laws (normalized count of eigenvalues in an interval) and the properties of the limiting distribution. If the graphs converge to some random infinite graph (in the sense of Benjamini-Schramm), an important question is also the nature of the spectrum (pure point, absolutely continuous, singularly continuous) of almost every limit graph. Concerning the eigenvectors, one asks whether they become \emph{localized} as the graph grows large (decay exponentially, have small support) or \emph{delocalized}. Eigenvector delocalization is measured by many criteria~: large support, uniform distribution over the graph, and also norm estimates. If the supremum norms of $\ell^2$-normalized eigenvectors decay fast as the graph gets large, this forces the entries of the vector to spread out. Bounds on the $\ell^p$ norms give further insight into the shape of the eigenvectors. 

To present our results, it is instructive to start with very simple graphs~: $N$-cycles. The set of vertices is $\{0,\dots,N-1\}$ and each point has two neighbors. The eigenvalues and eigenvectors of the adjacency matrix on such graphs are completely explicit~: we have
\[
\lambda_j = 2\cos\left(\frac{2j\pi}{N}\right) \quad \text{and} \quad \psi_{\lambda_j} = \frac{1}{\sqrt{N}}\left(1,\omega^j,\omega^{2j},\dots,\omega^{(N-1)j}\right)
\]
for $j=0,\dots,N-1$, where $\omega = e^{\frac{2\pi\ii}{N}}$ and the eigenvectors are $\ell^2$-normalized. This is the ideal delocalization one can hope for~: all eigenvectors are perfectly uniformly distributed on the graph, in the sense that $|\psi_{\lambda}(k)|^2=\frac{1}{N}$ for all $k$. Here, $\|\psi_{\lambda}\|_{\infty} = \frac{1}{|G|^{1/2}}$, $\|\psi_{\lambda}\|_p = \frac{1}{|G|^{\frac{1}{2}-\frac{1}{p}}}$ for any $p>2$ and each $\psi_{\lambda}$ has full support on the graph.

As a first step towards generalization, one can replace such $2$-regular graphs by general $(q+1)$-regular graphs. This was carried out in a series of papers. It is shown in \cite{ALM,BLL,A} that most eigenfunctions of such graphs are uniformly distributed in a certain sense, a property known as \emph{quantum ergodicity}. Lower bounds on the support were provided in \cite{BL}, where it is shown more generally that the eigenfunctions cannot concentrate on small sets. The recent paper \cite{BL17} provides norm estimates which read as $\|\psi_{\lambda}\|_p \lesssim \frac{1}{(\log_q |G|)^{1/2}}$ for all $p>2$, for generic regular graphs. All the preceding results apply to \emph{deterministic} graphs. In case of \emph{random} regular graphs of high degree, much better delocalization properties were obtained in \cite{BHY}. It is shown essentially that with high probability, all the ideal properties of $N$-cycles remain true, modulo logarithmic corrections.

In this paper, we investigate the eigenvectors of graphs which are more general in two respects~: we add potentials and consider irregular graphs. A quantum ergodicity result was recently established in this framework \cite{AS2,AS3,AS4}, where it is shown essentially that if a sequence of graphs converges in the sense of Benjamini-Schramm\footnote{Roughly speaking, $(G_N)$ converges to $G$ in the sense of Benjamini-Schramm if random $k$-balls in $G_N$ look similar to random $k$-balls in $G$, as $N$ gets large. See \cite{BS,AL} for details and adaptations to $(G_N,W_N)$.} to a random tree, and if the spectrum at the limit is purely absolutely continuous almost surely, then most eigenvectors become uniformly distributed in some sense. Our aim here is to provide norm estimates for the eigenfunctions and discuss their support. Our results hold in particular for \emph{all eigenfunctions} in the bulk spectrum, whereas quantum ergodicity is a different delocalization criterion which only assesses \emph{most eigenfunctions} in an interval.

Our theorems are formulated for fixed graphs $(G,W)$, where $W:V(G)\to \R$ is a potential and we consider the Schr\"odinger operator $H_G = \cA_G+W$, with $\cA_G$ the adjacency matrix. If one wants to consider the asymptotics of a sequence $(G_N,W_N)$, one must keep track of the corresponding constants in the inequalities, in a common energy region for all $N$. Our results do not cover all Benjamini-Schramm limits discussed in \cite{AS2}, see \S~\ref{sec:seq}. Still, we show that if $(G_N,W_N)$ is the set of $N$-lifts of an arbitrary finite graph $(G_1,W_1)$ with degree $\ge 2$, then our estimates allow to control the whole sequence. An $N$-lift $G_N$ is an $N$-cover over $G_1$ with $W_N(v) = W_1(\pi_N v)$, where $\pi_N :G_N\to G_1$ is the covering projection. The theory of random $N$-lift has been extensively studied in the last two decades, as a natural model of random irregular graphs \cite{AL02,ALM02,Fri03,BL06,Pu15,Bornew}. In particular, such graphs are typically connected (an assumption we make throughout). 

\subsection{Main results}
We shall use the notation $\sigma(A)$ for the spectrum of an operator $A$.

Our first result concerns Schr\"odinger operators on $N$-cycles $G=\{0,\dots,N-1\}$. Each digit $j$ is endowed with a potential $W_j$ and we consider $H_G = \cA_G + W$. Our estimates depend on the behavior of the ``lifted'' Schr\"odinger operator $H_{\Z}$ on $\Z$, which is the universal cover of $G$, with periodic potential $W_{j+kN}:=W_j$. The spectrum of $H_{\Z}$ is purely absolutely continuous and consists of at most $N$ bands (see Section~\ref{sec:cycles}). We say that $\lambda$ is in \emph{the bulk of the spectrum} if $\lambda$ is in the interior of such a band.

\begin{thm}\label{thm:cycles}
Let $\psi_{\lambda}$ be an eigenfunction of $H_G$ on the $N$-cycle, $\|\psi_{\lambda}\|_2=1$.

If $\lambda\notin \sigma(H_{\Z})$ and $\delta_{\lambda}=\dist(\lambda,\sigma(H_{\Z}))$, then
\[
\|\psi_{\lambda}\|_{\infty}^2 \le \frac{16\delta_{\lambda}^{-2}}{N} \,.
\]
If $\lambda$ is in the bulk of the spectrum, assume moreover that the potential is $m$-periodic. That is $N=N'm$ and $W_{j+km} = W_j$ on $G$. Then
\[
\|\psi_{\lambda}\|_{\infty}^2 \le \frac{C_{\lambda,m}}{N} \,.
\]
\end{thm}

This theorem shows that the ideal decay of the norms on $N$-cycles that we discussed before remains true if we add periodic potentials.

One can also ask about the support of such eigenfunctions. The answer is easy~: if $\psi_{\lambda}$ is an eigenfunction of $H_G$, then $\psi_{\lambda}(k+1) = (\lambda-W_k) \psi_{\lambda}(k) - \psi_{\lambda}(k-1)$. So if $\psi_{\lambda}$ vanishes on two consecutive digits, it vanishes identically. This shows the support is at least $\lceil \frac{N}{2} \rceil$. This lower bound is sharp in general.

The theorem remains true for weighted Schr\"odinger operators $(H_G \psi)(k)=a_k \psi(k+1) + a_{k-1}\psi(k-1) + W_k\psi(k)$, assuming $a_{j+m}=a_j$ in case of the bulk.

\bigskip

We now move to general graphs $G$, which we always assume to be connected. Our estimates are useful when the graph does not have too many short cycles. We thus define

\begin{itemize}
\item $\rho_G$, the minimal injectivity radius of $G$, i.e. the largest $\rho\in \N$ such that the ball $B_G(x,\rho)$ is a tree for any $x\in G$,
\item $\ell_G$, the largest $\ell\in \N$ such that $B_G(x,\ell)$ has at most one cycle for any $x\in G$.
\end{itemize}

Equivalently, $\rho_G$ is half the girth (length of the smallest cycle in $G$). By definition, $\ell_G \ge \rho_G$. In case of random $(q+1)$-regular graphs, we have $\ell_G \to +\infty$ almost surely as $|G|\to +\infty$. In fact, one has $\ell_G \ge \frac{1}{5} \log_q|G|$ almost surely, see \cite{MWW04}. This remains true for certain irregular graphs. More precisely, it follows from \cite[Lemma 24]{Bornew} that if we consider a random $N$-lift $G_N$ of a fixed finite graph $G_1$ of maximal degree $D$, then with probability converging to one as $N\to\infty$, we have $\ell_{G_N} \ge \frac{1}{5} \log_{D-1} N$. In contrast, the probability that $\rho_G \ge c \log_q|G|$ is very small, see \cite[Corollary 1]{MWW04}.

Henceforth we assume that $\ell_G$ is large.

\medskip

Our estimates show that eigenvector delocalization is intimately related to the behavior of the Green function on the universal cover of the graph. Let us introduce some notation. 

If $(\cT,\cW)$ is a tree with potential $\cW$ on the vertices, if $v,w\in \cT$, $v\sim w$ and $\gamma\in \C^+ = \{\Im z >0\}$, let
\[
\zeta_w^{\gamma}(v)=-(H_{\cT}^{(v|w)}-\gamma)^{-1}(v,v)
\]
be the Green function of $\cT^{(v|w)}\subset \cT$, the subtree containing $v$ if we remove the edge $(v,w)$ from $\cT$. Here $H_{\cT}^{(v|w)} = \cA_{\cT^{(v|w)}}+\cW_{\cT^{(v|w)}}$. If $(\cT,\cW)=(\widetilde{G},\widetilde{W})$ is the universal cover of $(G,W)$, we extend this definition to the finite graph by letting $\zeta_y^{\gamma}(x) := \zeta_{\tilde{y}}^{\gamma}(\tilde{x})$ for $(x,y)\in B(G)$, the set of oriented edges of $G$. Here, $\tilde{x}\sim\tilde{y}$ are lifts of $x,y$ on the universal cover.

We emphasize that if $(x,y)$ is a directed edge in $G$, $\zeta_y^{\gamma}(x)$ is thus a Green function of an infinite tree, not the Green function of the finite graph.

We show in Theorem~\ref{thm:specone} that if $G$ has minimal degree at least $2$, then the spectrum of its universal cover $(\widetilde{G},\widetilde{W})$ consists of bands of AC spectrum and possibly some eigenvalues.

\begin{defa}
Fix $(G,W)$. We say that an eigenvalue $\lambda$ of $H_G$ belongs to the \emph{bulk of the spectrum} if it lies in the interior of an AC band of $H_{\widetilde{G}}$.

If $\lambda$ is in the bulk spectrum and $s>1$, we define the parameters
\begin{equation}\label{e:zlam}
z_{\lambda}(G) = \min_{(x,y)\in B(G)} |\Im \zeta_y^{\lambda+\ii0}(x)| \,,
\end{equation}
\begin{equation}\label{e:zslam}
Z_{s,\lambda}(G) = \max_{(x_0,x_1)\in B(G)} \sum_{x_2\in \cN_{x_1}\setminus \{x_0\}} \frac{|\zeta_{x_0}^{\lambda+\ii0}(x_1)|^{2s}|\Im\zeta_{x_1}^{\lambda+\ii0}(x_2)|^s}{|\Im\zeta_{x_0}^{\lambda+\ii0}(x_1)|^s},
\end{equation}
where $\cN_{x_1}$ denotes the set of neighbors of $x_1$.
\end{defa}

For example, if $G$ is $(q+1)$-regular and $W\equiv 0$, then $\widetilde{G}=\T_q$ is the $(q+1)$-regular tree and the bulk spectrum is $(-2\sqrt{q},2\sqrt{q})$ --- also known as the \emph{tempered} spectrum. Here, $z_{\lambda} = \frac{\sqrt{4q-\lambda^2}}{2q}$ and $Z_{s,\lambda}=q^{1-s}$.

In general, it follows from Theorem~\ref{thm:specone} that $z_{\lambda}(G)$ and $Z_{s,\lambda}(G)$ are well-defined, with $z_{\lambda}(G)>0$. Let us briefly explain the meaning of these parameters.

Recall that for $\lambda$ in the AC spectrum, the density of the spectral measure of $H_{\cT}$ at a vertex $v$ is given by $\frac{1}{\pi}\Im G^{\lambda+\ii0}_{\cT}(v,v)$. In case of trees, there are recursive relations between the Green functions $G^{\lambda+\ii0}_{\cT}(v,v)$ of $\cT$ and the Green functions of the subtrees $\cT^{(v|w)}$, so $z_{\lambda}$ is related to the minimal spectral density at $\lambda$.

To understand $Z_{s,\lambda}$, note that for $s=1$, we actually have $Z_{1,\lambda}=1$. More precisely, $\sum_{x_2\in \cN_{x_1}\setminus \{x_0\}} \frac{|\zeta_{x_0}^{\lambda+\ii0}(x_1)|^{2}|\Im\zeta_{x_1}^{\lambda+\ii0}(x_2)|}{|\Im\zeta_{x_0}^{\lambda+\ii0}(x_1)|}=1$ for any $(x_0,x_1)$. This can be viewed as a kind of ``conservation of current'' relation akin to Kirchhoff's law : if we fix some origin $o\in\cT$ such that the edge $(\tilde{x}_0,\tilde{x}_1)$ descends from $o$ (i.e. there is a non-backtracking path $(o,v_1,\dots,v_m,\tilde{x}_0,\tilde{x}_1)$), then $\zeta_{x_0}^{\lambda+\ii0}(x_1) = \frac{G^{\lambda+\ii0}(o,\tilde{x}_1)}{G^{\lambda+\ii0}(o,\tilde{x}_0)}$. Consequently, if we define the ``current'' $I_o^{\lambda}(v,w) = |G^{\lambda+\ii0}(o,v)|^2|\Im \zeta^{\lambda+\ii0}_v(w)|$, we get $I_o^{\lambda}(\tilde{x}_0,\tilde{x}_1) = \sum_{\tilde{x}_2\in \cN_{\tilde{x}_1}\setminus\{\tilde{x}_0\}} I_o^{\lambda}(\tilde{x}_1,\tilde{x}_2)$. In other words, the current on $(\tilde{x}_0,\tilde{x}_1)$ is preserved when passing through the forward edges $(\tilde{x}_1,\tilde{x}_2)$. A key idea for our $\ell^p$-norm estimate Theorem~\ref{thm:psupp} is to observe that if $s>1$, then this implies $Z_{s,\lambda}<1$, if $\min \deg G \ge 3$. This essentially allows us to say that for arbitrary graphs $(G,W)$, the Green function of the universal cover $|G_{\cT}^{\lambda+\ii0}(v_0,v_r)|$ on an $r$-path $(v_0,\dots,v_r)$ decays exponentially in $r$ (cf .\eqref{e:inftyboun}) and $\|G^{\lambda+\ii0}\delta_v\|_{2s}< \infty$ for all $s>1$ (cf \eqref{e:greensno}), for all $\lambda$ in the bulk spectrum. These facts can be proved for regular trees with $W\equiv 0$ by explicit calculation.

\smallskip

We now state our first general result on supremum norms. For this we only need $z_{\lambda}$.

\begin{thm}\label{thm:main}
Let $(G,W)$ be a graph of  minimal degree $\ge 2$. Let $\psi_{\lambda}$ be an eigenfunction of $H_G$, $\|\psi_{\lambda}\|_2=1$.
\begin{enumerate}[\rm (1)]
\item If $\lambda$ is in the bulk spectrum, we have
\[
\|\psi_{\lambda}\|_{\infty} \le \frac{8Dz_{\lambda}(G)^{-4}}{\sqrt{\ell_G}} \,,
\]
where $D$ is the maximal degree of $G$.
\item If $\lambda \notin\sigma(H_{\widetilde{G}})$, we have the much better bound
\[
\|\psi_{\lambda}\|_{\infty} \le \frac{8D}{\delta_{\lambda}} \left(1+\frac{\delta_{\lambda}}{2D}\right)^{-\ell_G} ,
\]
where, $\delta_{\lambda}=\dist(\lambda,\sigma(H_{\widetilde{G}}))$.
\end{enumerate}
\end{thm}

The fact that eigenfunctions have better delocalization properties outside the spectrum of $H_{\widetilde{G}}$ was already known in case of $(q+1)$-regular graphs with $W=0$. This region corresponds to the ``untempered spectrum''. In this case, \cite[Lemma 3.1]{BL17} implies that
\begin{equation}\label{e:bl17}
\|\psi_{\lambda}\|_{\infty} \lesssim \frac{1}{q^{c_{\lambda} \ell_G}} \,.
\end{equation}
In particular, if $\ell_G \ge c \log_q|G|$, which holds for a typical random graph, this gives
\begin{equation}\label{e:untem}
\|\psi_{\lambda}\|_{\infty} \lesssim \frac{1}{|G|^{\varepsilon_{\lambda}}} \,.
\end{equation}
This very simple result somehow complements the bounds $\|\psi_{\lambda}\|_{\infty} \lesssim \frac{\log |G|}{|G|^{1/2}}$ proved for random graphs in \cite{BHY} for $\lambda$ in the \emph{bulk} of the spectrum (this said, \eqref{e:untem} is of course much easier to prove). In our case, we know we can still take $\ell_G = \frac{1}{5} \log_{D-1}|G|$ if $G$ is a random $N$-lift of a base graph $G_1$, yielding $\|\psi_{\lambda}\| \lesssim \frac{1}{|G|^{\frac{1}{5}\log_{D-1}(1+\frac{\delta_{\lambda}}{2D})}}$. In contrast, our estimates in the bulk yield $\|\psi_{\lambda}\|_{\infty} \lesssim \frac{1}{(\log_{D-1}|G|)^{1/2}}$ generically.

We point out that as the graph grows large, most eigenvalues will lie in $\sigma(H_{\widetilde{G}})$. More precisely, if for example $G_N$ is a random $N$-lift of $G_1$, it follows from the estimates in \cite[Lemma 24]{Bornew}, \cite[Lemma 9]{BDGHT} that $(G_N,W_N)$ converges to the universal cover $(\widetilde{G}_1,\widetilde{W}_1)$ in the sense of Benjamini-Schramm with high probability. It is known that Benjamini-Schramm convergence implies the convergence of spectral measures, so the observation follows.

\begin{rem}
In view of Theorem~\ref{thm:specone}, there are two finite sets of exceptional energies that are not considered in Theorem~\ref{thm:main}~: the set $\mathfrak{F}$ of infinitely degenerate eigenvalues of $H_{\widetilde{G}}$, and the set $\mathfrak{F}'$ of endpoints of the AC spectrum. Note that $\mathfrak{F}=\emptyset$ for $(q+1)$-regular graphs with $W=0$ and $\mathfrak{F}'=\{\pm 2\sqrt{q}\}$. We believe it is natural to exclude $\mathfrak{F}$ from our considerations; in fact we expect any eigenfunction $\psi_{\lambda}$ of $H_G$ with eigenvalue in $\mathfrak{F}$ to be \emph{localized} (see the end of this section for an example). On the other hand, the fact that we exclude $\mathfrak{F}'$ from our analysis may be an artefact of our method. For example, we know from \cite{BL17} that it is not necessary to exclude $\mathfrak{F}'$ if the graph is regular and $W\equiv 0$.
\end{rem}

\medskip

As an easy consequence of Theorem~\ref{thm:main}, one gets

\begin{cor}
\phantomsection
\label{cor:bl}
\begin{enumerate}[\rm (i)]
\item \emph{($p$-norms).} For $\lambda$ in the bulk and $p>2$,
\[
 \|\psi_{\lambda}\|_p \lesssim \frac{1}{(\ell_G)^{\frac{p-2}{2p}}}\,.
\]
For $\lambda\notin\sigma(H_{\widetilde{G}})$, we have
\[
\|\psi_{\lambda}\|_p \lesssim \left(1+\frac{\delta_{\lambda}}{2D}\right)^{-\frac{p-2}{p} \ell_G} \,.
\]
\item \emph{(Non-localization).} Let $\Lambda\subset G$ and suppose $\|\chi_{\Lambda} \psi_{\lambda}\|_2^2 \ge \varepsilon$.

If $\lambda$ is in the bulk of the spectrum, then $|\Lambda| \ge \ell_G \cdot \frac{z_{\lambda}(G)^8\varepsilon}{64D^2}$.

If $\lambda\notin \sigma(H_{\widetilde{G}})$, then $|\Lambda| \ge \frac{\delta_{\lambda}^2}{64D^2} (1+\frac{\delta_{\lambda}}{2D})^{2\ell_G} \varepsilon$.
\end{enumerate}
\end{cor}

This corollary is deduced from the supremum bounds. The results are satisfactory for $\lambda\notin\sigma(H_{\widetilde{G}})$. For $\lambda$ in the bulk however, the estimate on $p$-norms is not sharp. For example, we have $\|\psi_{\lambda}\|_p \lesssim \frac{1}{\sqrt{\ell_G}}$ for all $p>2$ in the regular case \cite{BL17} (where the implied constant depends on $p$). As for non-localization, much better lower bounds $|\Lambda| \gtrsim q^{c \varepsilon^2 \ell_G} \varepsilon^2$ were established for $(q+1)$-regular graphs in \cite{BL}.

We were able to obtain sharper $p$-norm bounds as in \cite{BL17}, but for $p>4$. The ingredients are the properties of $Z_{s,\lambda}$ discussed after \eqref{e:zslam}, a $TT^{\ast}$ analysis and Young's inequality, which requires us to take $s=p/4$. This is why we need $p>4$. One may hope to use these ideas but avoid Young's inequality to have an estimate for all $p>2$. As for non-localization, we can also answer the special case where $\Lambda$ is the support of $\psi_{\lambda}$~:

\begin{thm}\label{thm:psupp}
Let $(G,W)$ be a graph of minimal degree $\ge 3$. Let $\psi_{\lambda}$ be an eigenfunction of $H_G$, $\|\psi_{\lambda}\|_2=1$. Assume $\lambda$ is in the bulk spectrum of $H_G$. Then
\begin{enumerate}[\rm (1)]
\item \emph{($p$-norms)} For any $p>4$,
\[
\|\psi_{\lambda}\|_p \le \frac{8D z_{\lambda}(G)^{-5}}{1-Z_{p/4,\lambda}^{2/p}(G)} \cdot \frac{1}{\sqrt{\ell_G}} \,.
\]
\item \emph{(Support)} If $\Lambda$ is the support of $\psi_{\lambda}$, let $M_{\lambda}=Z_{2,\lambda}^{-1/4}>1$. Then
\[
|\Lambda| \ge \frac{1}{4D} M_{\lambda}^{\ell_G} \,.
\]
\end{enumerate}
\end{thm}

Recall that if $(G,W)=(G_N,W_N)$ is an $N$-lift of some $(G_1,W_1)$, then $\ell_G \ge \frac{1}{5} \log_{D-1} |G|$ with high probability. In this case we get that the support of any $\psi_{\lambda}$ satisfies
\[
|\Lambda| \ge \frac{1}{4D} |G|^{\alpha_{\lambda}} \,,
\]
with $\alpha_{\lambda} = \frac{1}{5\log_{M_{\lambda}}(D-1)} = \frac{\ln M_{\lambda}}{5\ln (D-1)} >0$. This shows that the support of eigenfunctions (whether in the bulk or outside $\sigma(H_{\widetilde{G}})$) is generically a fractional power of the number of vertices --- a large set.

\subsection{Sequences of graphs}\label{sec:seq}

The previous results hold for any fixed graph $(G,W)$, in the prescribed energy regions.

An important question is to study sequences of graphs $(G_N,W_N)$ that grow larger as $N\to \infty$. In this case, to ensure that $\|\psi_{\lambda}\|_p \lesssim \frac{1}{\sqrt{\ell_{G_N}}}$ and $|\supp \psi_{\lambda}| \gtrsim K_{\lambda}^{\ell_{G_N}}$ with $K_{\lambda}>1$, we must keep track of the implicit constants and show they are controlled as $N$ gets large. This is the reason why we presented explicit constants in Theorems~\ref{thm:main} and \ref{thm:psupp}.

If $(G_N,W_N)$ are $(q+1)$-regular graphs with $W\equiv 0$, the situation is very simple~: we have $\sigma(H_{\widetilde{G}_N}) = [-2\sqrt{q},2\sqrt{q}]$ for all $N$, $\zeta^{\lambda}_y(x) = \frac{\lambda-\ii\sqrt{4q-\lambda^2}}{2q}$ for any $(x,y)$, independently of $G_N$, is the unique solution of $q\zeta^2-\lambda\zeta+1=0$ with negative imaginary part. In particular, $z_{\lambda}(G_N) = \frac{\sqrt{4q-\lambda^2}}{2q}$ independently of $N$ and $Z_{s,\lambda}(G_N) = q q^{-s} = q^{-(s-1)}$. Hence, the constants here are independent of $N$ and we have $\|\psi_{\lambda}\|_p \lesssim \frac{1}{\sqrt{\ell_{G_N}}}$ and $|\supp \psi_{\lambda}| \gtrsim M_{\lambda}^{\ell_{G_N}}$ for the whole sequence.

A more interesting class of graphs we can consider is given by $N$-lifts $\{(G_N,W_N)\}$, with $(G_N,W_N)$ the $N$-lift of some base graph $(G_1,W_1)$ say on $r$ vertices, so that $G_N$ has $Nr$ vertices. In this case, all graphs have the same universal cover $(\widetilde{G}_1,\widetilde{W}_1)$. In particular, $\sigma(H_{\widetilde{G}_N}) = \sigma(H_{\widetilde{G}_1})$ for all $N$ and we may consider an interval of energy $I$, say in the common bulk. On the other hand, since $(G_N,W_N)$ covers $(G_1,W_1)$, $\zeta_y^{\gamma}(x) = \zeta_{y'}^{\gamma}(x')$ whenever $\pi_N(x,y)=\pi_N(x',y')$, where $\pi_N:G_N\to G_1$ is the covering map. Indeed, $\zeta_y^{\gamma}(x) = \zeta_{\pi_Ny}^{\gamma}(\pi_Nx)$, as both are defined by $\zeta_{\tilde{y}}^{\gamma}(\tilde{x})$ where $\tilde{x},\tilde{y}\in\widetilde{G}_1$ are lifts to the universal cover. Hence, the $\{\zeta_y^{\gamma}(x)\}_{(x,y)\in B(G_N)}$ are just those of $G_1$. There are at most $Dr$ of them. In particular, $z_{\lambda}(G_N)=z_{\lambda}(G_1)$ and $Z_{s,\lambda}(G_N)=Z_{s,\lambda}(G_1)$, and we have again $\|\psi_{\lambda}\|_p \lesssim \frac{1}{\sqrt{\ell_{G_N}}}$ and $|\supp \psi_{\lambda}| \gtrsim M_{\lambda}^{\ell_{G_N}}$ for all $N$.

An example, however, to which our results do not seem to apply, is a sequence of Anderson models $(G_N,W_N^{\omega})$, where $G_N$ are regular graphs with few cycles and $W_N^{\omega}(x):=\omega_x$, with $(\omega_x)_{x\in V(G_N)}$ some i.i.d. random variables. We thus have a Schr\"odinger operator $H_{G_N}^{\omega} = \cA_{G_N} + W_N^{\omega}$ for each $\omega=(\omega_x)_{x\in V(G_N)}$. See \cite{AW15,AS3} for some background on this model. It is not clear if the parameters $z_{\lambda}(G_N)$ and $Z_{s,\lambda}(G_N)$ are asymptotically well-behaved. Moreover, there is a distinct universal cover for each $N$, does the bulk spectrum remain asymptotically large, or are we excluding too many points in $\mathfrak{F}_N$ as $N$ grows big~? On the other hand, it is natural that our results do not apply directly in this case, since otherwise we would have a statement for \emph{all $\omega$} instead of an almost-sure statement.

We expect our results to be true \emph{in mean} for the weakly disordered Anderson model, almost surely (e.g. \emph{most} (not all) eigenfunctions in the \emph{limiting AC spectrum} satisfy $\|\psi_{\lambda}\|_p\lesssim \frac{1}{\sqrt{\ell_G}}$ a.s.). Here the ``limiting AC spectrum'' is the AC spectrum of the Benjamini-Schramm limit of the sequence, replacing the bulk spectrum considered here. This different statement however (which would be weaker but apply to more general models) is still open. We mention that there is a heated debate among physicists about the ergodicity versus multi-fractility of the weakly disordered Anderson model on \emph{random} regular graphs \cite{Altshuler1,Altshuler2,TMS,MC}. One of the questions is whether $\|\psi_{\lambda}\|_p \lesssim \frac{1}{|G|^{\frac{1}{2}-\frac{1}{p}}}$ with high probability. This problem is completely out of reach at the moment (at least to us).

\subsection{Further remarks}

\begin{itemize}
\item The proof of Theorem~\ref{thm:main} yields more precisely
\[
|\psi_{\lambda}(x)| \le \frac{8Dz_{\lambda}^{-4}}{\sqrt{\ell_G(x)}} \,,
\]
where $\ell_G(x)$ is the largest $\ell\in\N$ such that $B_G(x,\ell)$ has at most one cycle. So even when $\ell_G$ is small, we can control the regions of the graph in which $\ell_G(x)$ is large.
\item If we introduce $\cZ_{\lambda} = \max_{(x_0,x_1,x_2)} \frac{|\zeta_{x_0}^{\lambda}(x_1)|^2|\Im\zeta_{x_1}^{\lambda}(x_2)|}{|\Im\zeta_{x_0}^{\lambda}(x_1)|}$ for bulk value $\lambda$, we have $\cZ_{\lambda}<1$ if $\min \deg G \ge 3$, and $Z_{s,\lambda} \le \cZ_{\lambda}^{s-1}$, cf. \S~\ref{e:greenp}. In particular, Theorem~\ref{thm:psupp} implies $\|\psi_{\lambda}\|_p \le \frac{8Dz_{\lambda}^{-5}}{1-\cZ_{\lambda}^{\frac{p-4}{2p}}} \frac{1}{\sqrt{\ell_G}}$, and one may take $p\to \infty$ to deduce a new proof of Theorem~\ref{thm:main} part (1). Note however that this only works for $\deg G \ge 3$, while Theorem~\ref{thm:main} allows for $\deg G \ge 2$. Moreover the proof of Theorem~\ref{thm:main} yields additional ``local'' information, as explained in the previous point. 
\item In the special case of biregular graphs\footnote{bipartite graphs with two types of vertices $\bullet,\circ$ where $\bullet$ has $d_1$ neighbors $\circ$, and $\circ$ has $d_2$ neighbors $\bullet$.}, our estimate on the support gives more precisely $|\Lambda|\ge \frac{[(d_1-1)(d_2-1)]^{\ell_G/4}}{4d_1}$, if $d_1 \ge d_2$. In particular, if $G$ is $(q+1)$-regular, $|\Lambda|\ge \frac{q^{\ell_G/2}}{4(q+1)}$. This estimate is sharper than \cite{BL} (which gives $|\Lambda| \gtrsim q^{2^{-8}\ell_G}$) and the proof is simpler, but the result only holds for the support, not general $\Lambda$.
\item All the results of Theorem~\ref{thm:psupp} also hold if we replace $\ell_G$ by $n\ge \ell_G$, as long as $n$ satisfies the following condition~: for any $\beta>0$, there are at most $2^{\beta k}$ paths of length $k$ between two edges, for each $k\le n$.

This condition was used in \cite{BL17}. While for a given graph, it may occur that such $n>\ell_G$, note that we always have $n\le \log_{D-1} |G|$, cf. \cite{BL17}. So for generic graphs, it suffices to consider $\ell_G$.
\item All results can be adapted to adjacency matrices with weights, namely $(\cA_p \psi)(x)=\sum_{y\sim x} p_x(y)\psi(y)$ with $p_x(y)$ positive and symmetric. This can be regarded as putting colors $p(b)$ on the edges. The finite graphs now take the form $(G,W,p)$, and the covers $(\widetilde{G},\widetilde{W},\widetilde{p})$.
\item The fact that eigenfunction delocalization depends on the behavior of the Green function was not apparent in \cite{BL17}, but is actually well-known, see \cite{Gei,BHY,Bor,AS2} to name a few references. For example, note that if $\psi_j$ is an eigenfunction of $H_G$ on $G$ with corresponding eigenvalue $\lambda_j$, then for any $\eta>0$,
\begin{equation}\label{e:borgei}
|\psi_j(x)|^2 \le \eta \sum_{k=1}^n \frac{\eta}{(\lambda_k-\lambda_j)^2+\eta^2} |\psi_k(x)|^2 = \eta \Im g^{\lambda_j+\ii\eta}(x,x)\,,
\end{equation}
where $g^{\gamma}(x,x) = \langle \delta_x,(H_G-\gamma)^{-1}\delta_x\rangle$ is the Green function \emph{of the finite graph}. If $|G|=N$ and if one can take $\eta \approx \frac{1}{N}$, and guarantee that $|\Im g^{\lambda+\frac{1}{N}}(x,x)|$ stays bounded, this implies that $|\psi_j(x)|^2 \lesssim \frac{1}{N}$, which is the ideal delocalization one can hope to achieve. This idea is implemented in \cite{BHY} in the framework of random regular graphs, so that most of the proof is devoted to controlling the Green function $g^{\gamma}$ and show that it is close to the Green function $G^{\gamma}$ of a tree-like graph, which is well-behaved.

We use a different idea in this paper~: we deal \emph{directly} with Green functions on trees (see Section~\ref{sec:op}), and the bounds come from the exponential decay of $G^{\lambda_j}(x_0,x_k
)$ with respect to the length of the path $(x_0;x_k)$, not from the parameter $\eta$ above.

Finally, we believe the approach of \cite{Gei,Bor} can be used to obtain supremum bounds of the form $\|\psi_{\lambda}\|_{\infty} \lesssim \sqrt{\frac{\log \rho_G}{\rho_G}}$ for the graphs considered in this paper. Such bounds are weaker than Theorem~\ref{thm:main} however, mainly because $\ell_G$ is generically of order $\log |G|$, while $\rho_G$ is not.
\item Our upper bound is in terms of $\ell_G$. It is natural to ask if we can go beyond this and prove strong delocalization bounds as in \cite{BHY}. In our framework of \emph{deterministic} graphs, one cannot hope for an upper bound involving $|G|$ directly. The presence of a few short cycles can create completely localized eigenfunctions, no matter how ``nice'' the rest of $G$ is. In fact, as shown in \cite[Figure 1]{BHY}, already in the context of $3$-regular graphs, the presence of a subgraph composed of two inverted triangles creates an eigenfunction supported on two points, no matter how large $G$ is. This shows that an upper bound involving $\ell_G$ is natural, but leaves the question open whether a decay better than $\ell_G^{-1/2}$ can be achieved for bulk eigenfunctions.

One may understand the decay $\ell_G^{-1/2}$ heuristically by drawing an analogy between graphs and manifolds, in which the semiclassical parameter $\lambda\to+\infty$ is replaced by $|G|\to +\infty$, so that the size of the graph is somehow analogous to the eigenvalue. In this case, the logarithmic decay $\ell_G^{-1/2}$ we obtain is similar to the logarithmic improvement obtained by Hassel and Tacy \cite{HT15} in the context of $L^p$ norms of eigenfunctions on manifolds. In both problems, one faces a difficulty to propagate beyond a certain time. In our case, beyond $\ell_G$, the number of paths between two points can very quickly become too big. With this analogy in mind, it seems unlikely to go beyond $\ell_G^{-1/2}$ with the current technology.

In the context of \emph{random} graphs however, we expect much better bounds than the ones achieved here to hold with high probability. We believe the present paper could provide a good starting point to prove \emph{complete delocalization} as in \cite{BHY}, but now in the framework of \emph{random $N$-lifts} of possibly irregular graphs.
\end{itemize} 

\medskip

We conclude this section with an example of graphs in which eigenfunctions localize on small regions, justifying why we exclude certain eigenvalues of the universal cover.

Consider a $6$-cycle $(x_0,\dots,x_5,x_0)$. Let $\psi(x_0)=\psi(x_3)=\frac{1}{2}$, $\psi(x_1)=\psi(x_4)=\frac{-1}{2}$, $\psi(x_2)=\psi(x_5)=0$. This gives an eigenfunction with eigenvalue $-1$. Now consider a long segment $(y_1,\dots,y_{3m-1})$. We glue this segment to the cycle, such that $y_1 \sim x_2$ and $y_{3m-1} \sim x_5$. We extend $\psi$ by $0$ on this new graph. Then $\psi$ is still an eigenfunction, localized on the small cycle. Here, $\|\psi_{\lambda}\|_{\infty} = \frac{1}{2}$ although $\ell_G$ is arbitrarily large (in particular, $\|\psi_{\lambda}\|_{\infty} \gg \ell_G^{-1/2}$).

On the other hand, $-1$ is an eigenvalue of the universal cover. To see this, root the tree at $x_0$ for example. Let $f(\tilde{x}_0)=1$, $f(\tilde{x}_1)=-1$, $f(\tilde{x}_5)=0$. This defines $f$ at the root and its neighbors. For the $2$-sphere, let $f(\tilde{x}_2)=0$, $f(\tilde{y}_{3m-1})=\frac{-1}{2}$, $f(\tilde{x}_4) = \frac{-1}{2}$. For further spheres, we follow this rule~: on each line segment, we extend $f$ using the only possible rule, namely $a,-a,0,a,-a,0,\dots$. Each time the tree splits (this occurs precisely at $\tilde{x}_2$, $\tilde{x}_5$), we let $f(\tilde{x}_2)=f(\tilde{x}_5)=0$. If $f$ has value $c$ at the parent of $\tilde{x}_2$ or $\tilde{x}_5$, we let $f(v)= \frac{-c}{2}$ on the two children. Note there are only two types of line segments~: those of length $2$, those of length $3m-1$. We thus find in any case that if $M=3m-1$, so that $2\le M$,
\[
\|f\|^2 \le M + 4M \left|\frac{1}{2}\right|^2 + 8 M \left|\frac{1}{4}\right|^2 + \dots \le M \sum_{k=0}^{\infty} 2^{k+1} \cdot 2^{-2k} < \infty \,,
\]
so $f$ is indeed an eigenvector of $H_{\widetilde{G}}$ with eigenvalue $-1$.

This example shows that energies in $\mathfrak{F}$ \emph{can} correspond to localized eigenfunctions on the finite graph. It would be interesting to study if this is always the case.

\section{Linking the eigenfunctions to the Green function on the  covering tree}\label{sec:op}

For any graph $\mathbb{G}$ and $v\in V(\mathbb{G})$, we denote $\cN_v = \{w\in V(\mathbb{G}):w\sim v\}$ the set of nearest neighbors of $v$. Directed edges are denoted by $b=(v,w)$. We let $o(b) = v$ be the origin of $b$ and $t(b)=w$ the terminus of $b$.

Let $G$ be a finite graph and $\cA_G$ its adjacency matrix. We are interested in the eigenfunctions $(\psi_{\lambda})$ of Schr\"odinger operators $H_G = \cA_G + W$ on $G$. Our aim in this section is to derive a convenient representation of $\psi_{\lambda}$ using the Green functions of the universal cover. For this, we start by investigating the spectral properties of such covering trees.

Throughout this section, we assume that

\medskip

\textbf{(C1)} The finite graph $G$ has minimal degree at least $2$ and is not a cycle.

\medskip

Let $\cT$ be a tree. If $v,w\in \cT$, $v\sim w$, we denote $\cT^{(v|w)}\subset \cT$ the subtree obtained by removing the branch emanating from $v$ passing by $w$ (keeping $v\in \cT^{(v|w)}$). We also denote $\zeta_w^{\gamma}(v) = -(H_{\cT}^{(v|w)}-\gamma)^{-1}(v,v)$, the (negative of the) corresponding Green function for $\gamma\in\C^+ = \{ \Im z>0\}$. Recall that Green's functions have the important property of being Herglotz, i.e. analytic on $\C^+$ and taking $\C^+$ to itself. Hence,
\begin{equation}\label{e:herg}
|\Im \zeta_w^{\gamma}(v)| = -\Im \zeta_w^{\gamma}(v) \,.
\end{equation}

For future reference, we recall the following identities \cite{KS,Klein,AS} for the Green function on a tree~: for any $\gamma\in\C^+$, $v\in\cT$, $w\sim v$,
\begin{equation}\label{e:rec}
\frac{1}{G^{\gamma}(v,v)} = W(v) + \sum_{u\sim v}\zeta_v^{\gamma}(u) - \gamma \,, \qquad 
\frac{-1}{\zeta_w^{\gamma}(v)} = W(v) + \sum_{u\in \cN_v\setminus \{w\}} \zeta_v^{\gamma}(u) - \gamma \,,
\end{equation}
\begin{equation}\label{e:zetainv}
\frac{1}{\zeta_w^{\gamma}(v)}-\zeta_v^{\gamma}(w) = \frac{-1}{G^{\gamma}(v,v)} \,, \qquad \zeta_w^{\gamma}(v) = \frac{G^{\gamma}(v,v)}{G^{\gamma}(w,w)}\zeta_v^{\gamma}(w) \,,
\end{equation}
\begin{equation}\label{e:greenmulti}
G^{\gamma}(v_0;v_k) = G^{\gamma}(v_0,v_0) \zeta_{v_0}^{\gamma}(v_1)\cdots\zeta_{v_{k-1}}^{\gamma}(v_k) \qquad G^{\gamma}(v,w)=G^{\gamma}(w,v)
\end{equation}
for any non-backtracking path $(v_0;v_k)$ in a tree $\cT$.

\begin{rem}[Trees of finite cone type] \label{rem:cones}
In this remark, we explain the consequence of assumption \textbf{(C1)} on the universal cover.

If $\T$ is a tree, fix an origin $o\in \T$ and regard the rest of the tree as descending from $o$. So $o$ has $\cN_o^+$ children which are its neighbors in $\T$, and each $v\neq o$ has a set $\cN_v^+$ of children and a single parent $v_-$. A \emph{cone} in $\T$ is a subtree $\T^{(v|v_-)}$, i.e. a subtree descending from $v$. We say that $\T$ is a tree of finite cone type if there are finitely many non-isomorphic cones.

This definition may be extended to allow potentials $W:\T\to\R$. One regards $(\T,W)$ as a ``colored tree'', where each vertex $v\in\T$ has a color $W(v)$. In this case, we say $(\T,W)$ is a tree of finite cone type if there are finitely many non-isomorphic colored subtrees $\T^{(v|v_-)}_W$.

For example, the $(q+1)$-regular tree $\T_q$ has only $2$ cone types : the cone at the origin (which has $q+1$ children) and the other cones (each of which has $q$ children). However, for the colored $(\T_q,W)$ to be of finite cone type, the potential $W$ must in particular take finitely many values.

For our purposes, the trees to keep in mind are the universal covers of finite graphs $(G,W)$. The potential is lifted to $\cT=\widetilde{G}$ naturally by $\cW(v) := W(\pi v)$, where $\pi:\cT\to G$ is the covering map. It is easy to see that $(\cT,\cW)$ is then a tree of finite cone type. In fact, the cones $\cT_\cW^{(v|v_-)}$ and $\cT_\cW^{(w|w_-)}$ are isomorphic whenever $\pi(v_-,v) = \pi(w_-,w) = (x,y)\in B(G)$, where $B(G)$ is the set of oriented edges of $G$. This shows in fact that $(\cT,\cW)$ has at most $|B(G)|+1$ cone types (the $+1$ comes from the cone at the origin). 

For such a universal cover, we use a finite set of labels $\mathfrak{A} = \{1,\dots,m\}$. We label the cone at the origin by $1$ and the rest by $j\in \{2,\dots,m\}$, according to the type of the colored cone. Note that the origin has a distinct label (even if some cone is isomorphic to the cone at the origin). We then define a matrix $(M_{i,j})_{i,j\in\mathfrak{A}}$, where a vertex with label $i$ has $M_{i,j}$ children of label $j$.

Assuming \textbf{(C1)} holds, we know by \cite[Lemma 3.1]{OW07} that the non-backtracking matrix $\cB$ on $G$ is irreducible. This means that for any directed edges $(x_0,x_1)$ and $(y_0,y_1)$ in $G$, there is a non-backtracking path $(v_0,\dots,v_k)$ such that $(v_0,v_1) = (x_0,x_1)$ and $(v_{k-1},v_k) = (y_0,y_1)$. Since the directed edges of $G$ index the different cones, this means that on $(\cT,W)$, any colored cone $\cT_\cW^{(w|w_-)}$ appears as a descendant of any $\cT_\cW^{(v|v_-)}$, except for the cone at the origin which does not appear as it has a distinct label. Hence, in terms of the matrix $M=(M_{i,j})$, we see that \textbf{(C1)} implies

\medskip

\textbf{(C1')} We have $M_{1,1}=0$. Moreover, for any $k,l\in \{2,\dots,m\}$, there is $n=n(k,l)$ such that $(M^n)_{k,l} \ge 1$.

\medskip

The fact that $M_{1,1}=0$ is by definition of the labels, the rest is implied by \textbf{(C1)}.
\end{rem}

\begin{thm}\label{thm:specone}
Assume that \emph{\textbf{(C1)}} holds and let $(\cT,\cW)=(\widetilde{G},\widetilde{W})$. Then

\begin{enumerate}[\rm (i)]
\item The spectrum of $H_{\cT}$  consists of a finite union of intervals and points~:
\begin{equation}\label{e:specone}
\sigma(H_{\cT}) = (\cup_{r=1}^{\ell} J_r) \sqcup \mathfrak{F} \sqcup \mathfrak{F}' \,,
\end{equation}
where $\sqcup$ denotes the disjoint union. Here $J_r$ are open intervals, $\mathfrak{F}$ is a finite set of points and $\mathfrak{F}'$ is the set of endpoints of the intervals $\{J_r\}$.

The set $\cup_{r=1}^{\ell} J_r$ is never empty.
\item The limit $\zeta_w^{\lambda}(v) := \lim_{\eta \downarrow 0} \zeta_w^{\lambda+\ii\eta}(v)$ exists for $\lambda$ in $J_r$, and satisfies $|\Im \zeta_w^{\lambda}(v)|>0$ for all $(v,w)$.
\item The limit $G^{\lambda+\ii0}(v,w)$ exists for any $v,w\in\cT$ and $\lambda\in J_r$.
\item The spectrum is purely absolutely continuous in closed $I\subset J_r$.
\item If $\mathfrak{F}\neq\emptyset$, then the points in $\mathfrak{F}$ are eigenvalues of infinite multiplicity.
\end{enumerate}
\end{thm}
\begin{proof}
Claims (i), (ii) are proved in \cite[Section 4]{AS4}. There the case $W\equiv 0$ was considered, but the proof remains the same~: the fact that $(\cT,\cW)$ is the universal cover of $(G,W)$ is responsible for making the Green function of $H_{\cT}$ algebraic. Condition \textbf{(C1)} implies condition \textbf{(C1')} from Remark~\ref{rem:cones}. The role of the latter condition is to link the behaviors of all $\zeta_w^{\lambda}(v)$ (there are only finitely many distinct such $\zeta$, indexed as $\zeta_j^{\gamma}$, $j\in \mathfrak{A}$). The fact that we may take a disjoint union in \eqref{e:specone} is due to the fact that $\mathfrak{F}$ is by construction a set on which some $\zeta_j^{\lambda+\ii 0}$ has a pole. Such a $\lambda$ cannot be in $J_r$ by (ii), so it is either isolated from $J_r$, or an endpoint of $J_r$. 

The fact that $\cup_{r=1}^{\ell} J_r \neq \emptyset$, i.e. that $\sigma(H_{\cT})$ always has some continuous spectrum,  follows in particular from the results of \cite[\S 1.6]{BSV17}, because the tree $\cT=\widetilde{G}$ always has at least two topological ends under assumption \textbf{(C1)}. In that paper, the authors only consider $\cA_{\cT}$, but their proof remains the same for $H_{\cT}$. One simply replaces \cite[eq. (7)]{BSV17} saying that $\lambda f(w)=(\cA f)(w)$ by $(\lambda-W(w))f(w) = (\cA f)(w)$.

For (iii), first note that by \eqref{e:rec} and the fact that $\lambda$ and $W$ are real, we have $|G^{\lambda}(v,v)| = |\frac{1}{W(v)+\sum_{u\sim v}\zeta_v^{\lambda}(u)-\lambda}| \le \frac{1}{|\Im \zeta_v^{\lambda}(u)|} <\infty$ using (ii). We deduce the claim for general $w$ using \eqref{e:greenmulti}.

Item (iv) follows from (iii) and the continuity of $\lambda\mapsto G^{\lambda+\ii0}(v,v)$, which follows from \eqref{e:rec} and the continuity of $\lambda\mapsto \zeta^{\lambda+\ii0}_w(u)$. The latter is proved in \cite{AS4}.

We now prove (v). Since any $\lambda\in\mathfrak{F}$ is an isolated point of the spectrum, it is an eigenvalue. Hence, if $\mu_v(J) = \langle \delta_v,\chi_J(H_{\cT})\delta_v\rangle$ is the spectral measure at $v\in \cT$, we know that $\mu_v(\{\lambda\})>0$ for some $v\in \cT$. By \cite[Theorem 1.6]{Si95}, we have $\lim_{\eta\downarrow 0}(-\ii\eta)G^{\lambda+\ii \eta}(w,w) = \mu_w(\{\lambda\})$ for any $w\in \cT$.

If $\pi:\cT\to G$ is the covering projection, let $\Gamma$ be the group of automorphisms $\{g\}$ of $\cT$ such that $\pi \circ g = \pi$. This group acts freely on $\cT$ and $\cT / \Gamma \cong G$. Clearly, $H_{\cT}$ is invariant under the action of $\Gamma$. It follows that $G^{\gamma}(gv,gw) = G^{\gamma}(v,w)$ for any $g\in \Gamma$. In particular, $\mu_{gv}(J)=\mu_v(J)$ for any $g\in \Gamma$.

Showing that $\lambda$ has infinite multiplicity amounts to proving $\rank \chi_{(\lambda-\eps,\lambda+\eps)}(H_{\cT}) = +\infty$ for all $\eps>0$. If this was not that case, this operator would be of finite rank. In particular it would have a finite trace. But $\tr \chi_{(\lambda-\eps,\lambda+\eps)}(H_{\cT}) = \sum_{w\in\cT} \mu_w (\lambda-\eps,\lambda+\eps) \ge \sum_{g\in \Gamma} \mu_{g v}(\{\lambda\}) = \sum_{g\in \Gamma} \mu_v(\{\lambda\}) = +\infty$. This proves (v).
\end{proof}

The previous theorem shows that $z_{\lambda}:=z_{\lambda}(G)$ is well-defined and strictly positive, see \eqref{e:zlam}. We deduce some simple bounds. By \eqref{e:rec}, $|G^{\gamma}(v,v)| \le \frac{1}{|\Im [W(v) + \sum_{u\sim v}\zeta_v^{\gamma}(u)-\gamma]|}$. Since $\min \deg G \ge 2$, it follows that for $\lambda$ in the bulk spectrum,
\begin{equation}\label{e:zboun1}
|G^{\lambda}(v,v)| \le \frac{1}{2z_{\lambda}} \quad \text{and}\quad |\zeta_w^{\lambda}(v)| \le \frac{1}{z_{\lambda}} \,,
\end{equation}
where we used the second part of \eqref{e:rec} for the other inequality. Similarly, using $|\zeta^{\lambda}|\ge z_{\lambda}$ and \eqref{e:zetainv}, we get
\begin{equation}\label{e:zboun2}
\frac{1}{|G^{\lambda}(v,v)|} \le 2 z_{\lambda}^{-1} \quad\text{and}\quad \frac{|\Im \zeta_w^{\lambda}(v)|}{|\zeta_w^{\lambda}(v)|^2} \le z_{\lambda}^{-1} \,,
\end{equation}
where we use $\frac{|\Im \zeta^{\lambda}|}{|\zeta^{\lambda}|^2} \le |\frac{1}{\zeta^{\lambda}}|$ in the second part.

\medskip

We now show how we link the behavior of the eigenfunctions of the finite graph to that of the Green function of the universal cover. This can be seen as an alternative to \eqref{e:borgei}. 

Henceforth, $\cT$ will always denote the universal cover $(\widetilde{G},\widetilde{W})$ of $(G,W)$.

\begin{prp}\label{prp:retour}
Let $H_G\psi_{\lambda} = \lambda \psi_{\lambda}$. Suppose $\lambda\in\R \setminus (\mathfrak{F}\cup \mathfrak{F}')$. Denote $G^{\lambda}(v,w) = \lim_{\eta\downarrow 0}G^{\lambda+\ii\eta}_{\cT}(v,w)$. Given $x_0\in G$, fix any $x_1\sim x_0$. Then for any $r,k\ge 1$,
\begin{multline}\label{e:outrep}
\psi_{\lambda}(x_0) = \sum_{(x_2;x_{r+1})} \left[G^{\lambda}(\tilde{x}_0,\tilde{x}_{r+1})\psi_{\lambda}(x_r) - G^{\lambda}(\tilde{x}_0,\tilde{x}_r)\psi_{\lambda}(x_{r+1}) \right] \\
+ \sum_{(x_{-k};x_{-1})}\left[G^{\lambda}(\tilde{x}_0,\tilde{x}_{-k})\psi_{\lambda}(x_{-k+1}) - G^{\lambda}(\tilde{x}_0,\tilde{x}_{-k+1})\psi_{\lambda}(x_{-k})\right] ,
\end{multline}
where the sums run over all paths $(x_2,\dots,x_{r+1})$ (resp. $(x_{-k},\dots,x_{-1})$) such that $x_2\in \cN_{x_1}\setminus \{x_0\}$ (resp $x_{-1}\in\cN_{x_0}\setminus\{x_1\}$). Here $\tilde{x}_m\in \cT$ is a lift of $x_m\in G$ with $d_{\cT}(\tilde{x}_0,\tilde{x}_m)=m$.

If $\lambda$ is in the bulk $\cup_{r=1}^{\ell}J_r$, we also have
\begin{multline}\label{e:bulkrep}
\psi_{\lambda}(x_0) = \frac{1}{|\Im\zeta_{x_0}^{\lambda}(x_1)|}\sum_{(x_2;x_{r+1})}\big[\Im\left(\zeta_{x_0}^{\lambda}(x_1)\cdots\zeta_{x_{r-1}}^{\lambda}(x_r)\right)\psi_{\lambda}(x_{r+1}) \\
- \Im\left(\zeta_{x_0}^{\lambda}(x_1)\cdots\zeta_{x_r}^{\lambda}(x_{r+1})\right)\psi_{\lambda}(x_r)\big]\,.
\end{multline}
\end{prp}

In the above representation, it may happen that $x_r=x_0$ (if there are cycles near $x_0$), but the corresponding point $\tilde{x}_r$ which appears in the Green function will always satisfy $d_{\cT}(\tilde{x}_0,\tilde{x}_r)=r$. The path in $\cT$ which links $\tilde{x}_0$ to $\tilde{x}_r$ is precisely $(\tilde{x}_0,\tilde{x}_1,\dots,\tilde{x}_r)$.

Note that we obtain directly the Green function on the tree, in contrast to \eqref{e:borgei} where one obtains the Green function on the finite graph and then approximates it to a treelike graph. The proof will use the idea of non-backtracking eigenfunctions first introduced in \cite{A} and further developed in \cite{AS2}, along with a simple observation \eqref{e:backnow}. The argument would be a bit simpler if we assumed moreover that $G^{\lambda}(\tilde{x}_0,\tilde{x}_0) \neq 0$. In this case, one can work directly with $\lambda\in\R$ instead of considering $\gamma=\lambda+\ii\eta$ as we do below.
\begin{proof}
Fix $\eta>0$, denote $\gamma=\lambda+\ii\eta$ and define
\begin{equation}\label{e:fftoile}
f_{\gamma}(x_0,x_1) = \psi_{\lambda}(x_1)-\zeta_{x_0}^{\gamma}(x_1)\psi_{\lambda}(x_0) \quad \text{and} \quad f_{\gamma}^{\ast}(x_0,x_1)=\psi_{\lambda}(x_0)-\zeta_{x_1}^{\gamma}(x_0)\psi_{\lambda}(x_1) \,.
\end{equation}
Then
\[
f_{\gamma}(x_0,x_1) + \frac{1}{\zeta_{x_1}^{\gamma}(x_0)} f_{\gamma}^{\ast}(x_0,x_1) = \left[ \frac{1}{\zeta_{x_1}^{\gamma}(x_0)} - \zeta_{x_0}^{\gamma}(x_1) \right] \psi_{\lambda}(x_0) \,.
\]
Hence, using \eqref{e:zetainv},
\begin{equation}\label{e:backnow}
\psi_{\lambda}(x_0) = -G^{\gamma}(\tilde{x}_0,\tilde{x}_0) \left[f_{\gamma}(x_0,x_1) + \frac{1}{\zeta_{x_1}^{\gamma}(x_0)}f_{\gamma}^{\ast}(x_0,x_1)\right] .
\end{equation}
Consider the non-backtracking operator $\cB$ on $\ell^2(B)$ defined by
\begin{equation}\label{e:nonbackoper}
(\cB f)(b) = \sum_{b^+\in \cN_b^+} f(b^+) \,,
\end{equation}
where $\cN_b^+$ is the set of outgoing edges from $b$, i.e. the set of $b'$ such that $t(b) = o(b')$ and $b'\neq \iota b$, where $\iota b$ is the edge reversal of $b$. Using that $H\psi_{\lambda} = \lambda \psi_{\lambda}$ along with the recursive identity \eqref{e:rec}, we find that
\begin{align*}
(\cB f_{\gamma})(x_0,x_1) &= \left[\lambda\psi_{\lambda}(x_1)-W(x_1)\psi_{\lambda}(x_1)-\psi_{\lambda}(x_0)\right] - \psi_{\lambda}(x_1)\left[\gamma-W(x_1)-\frac{1}{\zeta_{x_0}^{\gamma}(x_1)}\right]\\
& = \frac{1}{\zeta_{x_0}^{\gamma}(x_1)}f_{\gamma}(x_0,x_1)-\ii\eta\,\psi_{\lambda}(x_1) \,.
\end{align*}
Hence, $\zeta^{\gamma}\cB f_{\gamma} = f_{\gamma} - \ii\eta \zeta^{\gamma}\tau_+\psi$, where $(\tau_+\psi)(x_0,x_1)=\psi(x_1)$ and $\zeta^{\gamma}$ is the multiplication operator by $\zeta^{\gamma}(x_0,x_1)=\zeta_{x_0}^{\gamma}(x_1)$. By induction, we get
\begin{equation}\label{e:indu1}
(\zeta^{\gamma}\cB)^rf_{\gamma} = f_{\gamma}-\ii\eta\sum_{t=1}^r(\zeta^{\gamma}\cB)^{t-1}\zeta^{\gamma}\tau_+\psi_{\lambda} \,.
\end{equation}
Similarly, if $\iota \zeta^{\gamma}$ is the multiplication operator by $\iota\zeta^{\gamma}(x_0,x_1)=\zeta_{x_1}^{\gamma}(x_0)$, we get
\begin{equation}\label{e:indu2}
(\iota\zeta^{\gamma}\cB^{\ast})^kf_{\gamma}^{\ast} = f_{\gamma}^{\ast} - \ii\eta\sum_{t=1}^k(\iota\zeta^{\gamma}\cB^{\ast})^{t-1}\iota\zeta^{\gamma}\tau_-\psi_{\lambda} \,,
\end{equation}
where $(\tau_-\psi)(x_0,x_1)=\psi(x_0)$. On the other hand, for any $f\in\ell^2(B)$,
\[
\left[\left(\zeta^{\gamma}\cB\right)^rf\right](x_0,x_1) = \sum_{(x_2;x_{r+1})} \zeta_{x_0}^{\gamma}(x_1)\cdots\zeta_{x_{r-1}}^{\gamma}(x_r) f(x_r,x_{r+1}) \,,
\]
\[
\left[\left(\iota \zeta^{\gamma}\cB^{\ast}\right)^kf\right](x_0,x_1) = \sum_{(x_{-k};x_{-1})} \zeta_{x_1}^{\gamma}(x_0)\cdots\zeta_{x_{-k+2}}^{\gamma}(x_{-k+1}) f(x_{-k},x_{-k+1}) \,.
\]
Hence, inserting \eqref{e:indu1} and \eqref{e:indu2} into \eqref{e:backnow} and using \eqref{e:greenmulti}, we get
\begin{multline*}
\psi_{\lambda}(x_0)=-\bigg[\sum_{(x_2;x_{r+1})} G^{\gamma}(\tilde{x}_0,\tilde{x}_r)f_{\gamma}(x_r,x_{r+1}) + \sum_{(x_{-k};x_{-1})} G^{\gamma}(\tilde{x}_0,\tilde{x}_{-k+1})f_{\gamma}^{\ast}(x_{-k},x_{-k+1}) \\
+\ii\eta\sum_{t=1}^r \sum_{(x_2;x_t)} G^{\gamma}(\tilde{x}_0,\tilde{x}_t)\psi_{\lambda}(x_t) +\ii\eta\sum_{t=1}^k\sum_{(x_{-t+1};x_{-1})} G^{\gamma}(\tilde{x}_0,\tilde{x}_{-t+1})\psi_{\lambda}(x_{-t+1}) \bigg] \,.
\end{multline*}
Finally, we use \eqref{e:fftoile} to expand the first two sums and take $\eta \downarrow 0$. Since $\lambda\in\R\setminus (\mathfrak{F}\cup\mathfrak{F}')$, we know from Theorem~\ref{thm:specone} that all $G^{\lambda}(v,w)$ exist. In particular, the last two error sums vanish, proving the first representation.

For the second one, let
\begin{equation}\label{e:fglambda}
f_{\lambda}(x_0,x_1) = \psi_{\lambda}(x_1) - \zeta_{x_0}^{\lambda}(x_1) \psi_{\lambda}(x_0) \quad \text{and} \quad g_{\lambda}(x_0,x_1) = \psi_{\lambda}(x_1) - \overline{\zeta_{x_0}^{\lambda}(x_1)}\psi_{\lambda}(x_0) \, .
\end{equation}
Then for $\lambda\in J_r$,
\begin{equation}\label{e:psidif}
\psi_{\lambda}(x_0) = \frac{f_{\lambda}(x_0,x_1)-g_{\lambda}(x_0,x_1)}{2\ii\,|\Im \zeta^{\lambda}_{x_0}(x_1)|} \,.
\end{equation}
On the other hand, $\cB f_{\lambda} = \frac{1}{\zeta^{\lambda}}f_{\lambda}$, $\cB g_{\lambda} = \frac{1}{\overline{\zeta^{\lambda}}} g_{\lambda}$. Hence,
\[
\psi_{\lambda}(x_0) = \frac{[(\zeta^{\lambda}\cB)^rf_{\lambda}](x_0,x_1) - [(\overline{\zeta^{\lambda}}\cB)^rg_{\lambda}](x_0,x_1)}{2\ii\,|\Im\zeta_{x_0}^{\lambda}(x_1)|} \,.
\]
Expanding this gives the second representation.
\end{proof}

\section{The proof for supremum norms}

We begin the section by analyzing balls $B_G(x,n)$ with $n\le \ell_G$.

\begin{rem}\label{rem:tangle}
If $n\le \rho_G$, there is a single path from $x$ to $y\in B_G(x,n)$. If $\rho_G<n\le \ell_G$, there can be more. If $d_G(x,y)=k\le \ell_G$, because there is at most one cycle in $B_G(x,n)$, it is easy to see that there are at most two paths of length $k$ from $x$ to $y$.

Paths of length $k'>k$ can also reach $y$ by winding around some cycle $C\subset B_G(x,n)$ before terminating at $y$. The path may loop several times on $C$, but once it has exited the cycle, it cannot go back to it. In fact, after leaving $C$, the path has a unique road to reach $y$. If it were to come back to $C$, it would have to backtrack on this road, which we exclude. If we fix the length $k'$, there can be at most two such paths~: those traversing the cycle from either direction.

If $y$ is on a cycle, there can also be a path of length $k''>k$ which does not wind, but traverses the cycle in opposite direction (this situation will not arise in our proof later).
\end{rem}

\subsection{Within the bulk}\label{sec:bulk}

Recall the non-backtracking operator $\cB$ defined in \eqref{e:nonbackoper}. Given $\lambda\in J_r$, we now consider the operators on $\ell^2(B)$ defined by
\[
\cM_{n,\lambda} = \frac{1}{n}\sum_{r=1}^n\frac{1}{|\Im \zeta^{\lambda}|^{1/2}}\left(\zeta^{\lambda}\cB\right)^r|\Im\zeta^{\lambda}|^{1/2}
\]
and $\overline{\cM_{n,\lambda}} = \frac{1}{n}\sum_{r=1}^n\frac{1}{|\Im \zeta^{\lambda}|^{1/2}}(\overline{\zeta^{\lambda}}\cB)^r|\Im\zeta^{\lambda}|^{1/2}$. These will replace the ``spectral cluster operators'' $\frac{1}{N}W_{N,\alpha}$ considered in \cite{BL17}, see also \cite{HT15}. As in these references, we aim to take $n$ of logarithmic size, namely $\ell_G$. There are notable differences here however. In \cite{BL17,HT15}, the spectral cluster is essentially a function of the Laplacian. Here $\cM_{n,\lambda}$ is not normal, and in fact defined on a different Hilbert space~: $\ell^2(B)$ instead of $\ell^2(V)$. In case the graph is regular and $W\equiv 0$, one can use the explicit basis $(h_{\lambda_j})$ of $\ell^2(B)$ given in \cite[Section 7]{ALM} to see that $\cM_{n,\lambda_j} h_{\lambda_j}=h_{\lambda_j}$ and $\cM_{n,\lambda_j} h_{\lambda_k} = O(\frac{1}{n}) h_{\lambda_k}$ on the other basis elements, so that we have indeed a kind of ``projection'' onto the eigenfunction $h_{\lambda_j}$ of $\cB$. It is not clear if this remains true for general graphs. But the $\cM_{n,\lambda_j}$ do preserve $h_{\lambda_j}$ in full generality. More precisely, recall the functions $f_{\lambda}$, $g_{\lambda}$ in \eqref{e:fglambda}. As $(\zeta^{\lambda}\cB)^r f_{\lambda} = f_{\lambda}$ and $(\overline{\zeta^{\lambda}}\cB)^r g_{\lambda}=g_{\lambda}$ we get $\cM_{n,\lambda} \frac{f_{\lambda}}{|\Im \zeta^{\lambda}|^{1/2}} = \frac{f_{\lambda}}{|\Im\zeta^{\lambda}|^{1/2}}$ and $\overline{\cM_{n,\lambda}} \frac{g_{\lambda}}{|\Im \zeta^{\lambda}|^{1/2}} = \frac{g_{\lambda}}{|\Im\zeta^{\lambda}|^{1/2}}$. Recalling \eqref{e:psidif}, we thus get
\begin{equation}\label{e:rebulk}
\psi_{\lambda}(x_0) = \frac{1}{2\ii\,|\Im\zeta^{\lambda}_{x_0}(x_1)|^{1/2}}\left[ \left(\cM_{n,\lambda} \frac{f_{\lambda}}{|\Im \zeta^{\lambda}|^{1/2}}\right)(x_0,x_1) -  \left(\overline{\cM_{n,\lambda}} \frac{g_{\lambda}}{|\Im \zeta^{\lambda}|^{1/2}}\right)(x_0,x_1) \right]
\end{equation}
for any $x_1\sim x_0$.

We estimate the first sum, the other is similar. Denote $b_j=(x_{j-1},x_j)$. Write
\[
\left(\cM_{n,\lambda} \frac{f_{\lambda}}{|\Im \zeta^{\lambda}|^{1/2}}\right)(b_1) = \sum_{b'\in B} \cM_{n,\lambda}(b_1,b') \frac{f_{\lambda}}{|\Im \zeta^{\lambda}|^{1/2}}(b') \,.
\]

To describe the kernel $\cM_{n,\lambda}(b_1,b')$ above, let $\cB^r b_1$ be the set of edges which can be reached by a non-backtracking path of length $r$ from $b_1$. For example $\cB b_1 = \cN_{b_1}^+$ is the set of edges $(x_1,x_2)$ as $x_2$ varies over $\cN_{x_1}\setminus \{x_0\}$. Then we may replace sums over $(x_2;x_{r+1})$ by sums over $b_{r+1}\in \cB^r b_1$.

Suppose first the injectivity radius at $x_0$ satisfies $\rho_G(x_0)\ge n$. If $b_{r+1}\in \cB^rb_1$, there is a single path $(b_1,b_2,\dots,b_{r+1})$ form $b_1$ to $b_{r+1}$.

Let $\cK(b_1,b_2) = \frac{\zeta(b_1)|\Im\zeta(b_2)|^{1/2}}{|\Im\zeta(b_1)|^{1/2}}$ for $b_2\in \cB b_1$, $\cK(b_1,b_3) = \frac{\zeta(b_1)\zeta(b_2)|\Im\zeta(b_3)|^{1/2}}{|\Im\zeta(b_1)|^{1/2}}$ for $b_3\in\cB^2 b_1$. More generally, for $r\le n$ we let $\cK(b_1,b_{r+1}) = \frac{\zeta(b_1)\cdots\zeta(b_r)|\Im\zeta(b_{r+1})|^{1/2}}{|\Im\zeta(b_1)|^{1/2}}$ for $b_{r+1}\in\cB^r b_1$, and $\cK(b_1,b')=0$ otherwise. Then we have precisely $\cM_{n,\lambda}(b_1,b') = \frac{1}{n}\cK(b_1,b')$. Now by the Cauchy-Schwarz inequality,
\begin{equation}\label{e:cauchy-schwarz}
\left|\left(\cM_{n,\lambda} \frac{f_{\lambda}}{|\Im \zeta^{\lambda}|^{1/2}}\right)(b_1)\right| \le \left(\sum_{b'\in B} \left|\cM_{n,\lambda}(b_1,b')\right|^2\right)^{1/2} \left\|\frac{f_{\lambda}}{|\Im \zeta^{\lambda}|^{1/2}}\right\|_2 \,.
\end{equation}
On the other hand,
\begin{align*}
\sum_{b'\in B} \left|\cM_{n,\lambda}(b_1,b')\right|^2 & = \frac{1}{n^2}\sum_{r=1}^n\sum_{b_{r+1}\in \cB^rb_1} \frac{|\zeta^{\lambda}(b_1)\cdots\zeta^{\lambda}(b_r)|^2|\Im\zeta^{\lambda}(b_{r+1})|}{|\Im \zeta^{\lambda}(b_1)|} \\
& = \frac{1}{n^2}\sum_{r=1}^n\sum_{(x_2;x_{r+1})} \frac{|\zeta^{\lambda}_{x_0}(x_1)\cdots\zeta^{\lambda}_{x_{r-1}}(x_r)|^2|\Im\zeta^{\lambda}_{x_r}(x_{r+1})|}{|\Im \zeta^{\lambda}_{x_0}(x_1)|} \,.
\end{align*}
It follows from \eqref{e:herg} and \eqref{e:rec} that
\begin{equation}\label{e:magicim}
\sum_{x_{k+1}\in \cN_{x_k}\setminus\{x_{k-1}\}} |\Im \zeta^{\lambda}_{x_k}(x_{k+1})| = \frac{|\Im \zeta^{\lambda}_{x_{k-1}}(x_k)|}{|\zeta^{\lambda}_{x_{k-1}}(x_k)|^2} \,.
\end{equation}
Using this identity repeatedly, we deduce that
\begin{equation}\label{e:recurpath}
\sum_{(x_2;x_{r+1})} |\zeta^{\lambda}_{x_0}(x_1)\cdots\zeta^{\lambda}_{x_{r-1}}(x_r)|^2|\Im\zeta^{\lambda}_{x_r}(x_{r+1})| = |\Im\zeta_{x_0}^{\lambda}(x_1)| \,.
\end{equation}
Hence, $\sum_{b'\in B} |\cM_{n,\lambda}(b_1,b') |^2 = \frac{1 }{n}$ (assuming $n\le \rho_G(x_0)$).

\medskip

Now assume more generally that $n \le \ell_G$ and suppose $B_G(x_0,n)$ contains a cycle $C=(u_0,\dots,u_m)$. If $x_0\notin C$, then there is a unique path $\mathfrak{p}=(x_0,y_1,\dots,y_k,u_i)$ from $x_0$ to a point $u_i\in C$, such that $\mathfrak{p}\subset B_G(x_0,n)$ and $y_j\notin C$ for all $j$. Indeed, as $n\le \ell_G$, we know $C$ is the only cycle in $B_G(x_0,n)$. If there were two paths $(x_0,y_1,\dots,y_k,u_i)$ and $(x_0,y_1',\dots,y_r',u_j)$ from $x_0$ to points $u_i,u_j\in C$, we would get an additional cycle besides $C$ in the walk $(x_0,y_1,\dots,y_k,u_i,u_{i-1},\dots,u_j,y_r',y_{r-1}',\dots,x_0)$, which is forbidden\footnote{If the two paths only intersect at $x_0$, this walk is a cycle. In general, if $u_i\neq u_j$, let $y_s$ be the last vertex with $y_{s-1}=y_{s-1}'$. Then a cycle is formed by $(y_{s-1},y_s,\dots,y_k,u_i,u_{i-1},\dots,u_j,y_r',y_{r-1}',\dots,y_s',y_{s-1})$. If $u_i=u_j$, it is meant to consider $(x_0,y_1,\dots,y_k,u_i,y_r',y_{r-1}',\dots,x_0)$. In this case, let $y_s$ be the first vertex with $y_s\neq y_s'$ and $y_t$ the first vertex with $t>s$ such that $y_t=y_t'$ (if $y_t\neq y_t'$ for all $t>s$, let $y_t=y_t':=u_i$). Then a cycle is formed by $(y_{s-1},y_s,\dots,y_t,y_{t-1}',\dots,y_s',y_{s-1})$. In all cases, the cycles contain vertices $x_0$ or $y_m \notin C$ and are thus distinct from $C$.}.

Now the nice property is that \eqref{e:rebulk} is valid for any $x_1\sim x_0$. If $(x_0,y_1,\dots,y_k,u_i)$ is the unique path from $x_0$ to $C$, we simply choose $x_1\neq y_1$. Then we are sure that the paths $(x_2;x_{r+1})$ outgoing from $(x_0,x_1)$ will never meet the cycle $C$. Hence, in this case the values of $\cM_{n,\lambda}(b_1,b')$ are the same as the case $\rho_G(x_0)\ge n$ and we get again $\sum_{b'\in B} |\cM_{n,\lambda}(b_1,b') |^2 = \frac{1 }{n}$.

Now suppose $x_0\in C$. If $\deg(x_0)\ge 3$, we use the same idea~: say $x_0=u_0$. Then we choose $x_1 \sim x_0$ to be a neighbor not in $C$; i.e. $x_1\neq u_1,u_m$. Then again we are sure the paths $(x_2;x_{r+1})$ outgoing from $(x_0,x_1)$ will never meet the cycle, otherwise we would get at least two cycles in $B_G(x_0,n)$. So we get again $\sum_{b'\in B} |\cM_{n,\lambda}(b_1,b') |^2 = \frac{1 }{n}$.

\medskip

The case $x_0\in C$ and $\deg(x_0)=2$ is more subtle. Here we can no longer avoid the cycle $C=(u_0,\dots,u_m)$. Say $x_0=u_0$ and let $x_1=u_1$, so that $b_1=(u_0,u_1)$. Consider $b_{r+1}\in \cB^r b_1$, say $b_{r+1}\notin C$. While in the previous cases, $\frac{f_{\lambda}}{|\Im \zeta^{\lambda}|^{1/2}}(b_{r+1})$ was simply multiplied by $\frac{\zeta(b_1)\cdots\zeta(b_r)|\Im \zeta(b_{r+1})|^{1/2}}{n\,|\Im \zeta(b_1)|^{1/2}}$, here elements from $\cB^s b_1$, $s>r$ can also reach the same point $b_{r+1}$ and add a contribution to $\frac{f_{\lambda}}{|\Im \zeta^{\lambda}|^{1/2}}(b_{r+1})$. This occurs precisely when a long path first makes a few loops around $C$, then leaves the cycle. As mentioned in Remark~\ref{rem:tangle}, once it left $C$, it cannot return to it.

Similarly, any $b_{r+1}\in C$ can be reached through many paths~: the shortest one of length $r$, then the longer ones of length $s=r+mk \le n$, which first make $k$ loops around $C$.

From these considerations, we see that for $e_{r+1}\in \cB^r b_1$, $e_{r+1}=(u_r,u_{r+1})$,
\[
\cM_{n,\lambda}(b_1,e_{r+1}) = \frac{1}{n}\left(\sum_{j=0}^{k_r} \left(\zeta_{u_0}^{\lambda}(u_1)\cdots\zeta_{u_m}^{\lambda}(u_0)\right)^j\right)\frac{\zeta_{u_0}^{\lambda}(u_1)\cdots\zeta_{u_{r-1}}^{\lambda}(u_r)|\Im\zeta_{u_r}^{\lambda}(u_{r+1})|^{1/2}}{|\Im \zeta_{u_0}^{\lambda}(u_1)|^{1/2}} \,,
\]
with $\cM_{n,\lambda}(b_1,b_1) = \frac{1}{n}\sum_{j=1}^{k_0} (\zeta_{u_0}^{\lambda}(u_1)\cdots\zeta_{u_m}^{\lambda}(u_0))^j$. If $b_{r+1}\in \cB^rb_1$, $b_{r+1}\notin C$, let $u_i$ be the unique vertex in $C$ of degree $\ge 3$ which is closest to $o(b_{r+1})$. Then
\begin{multline}\label{e:outbr}
\cM_{n,\lambda}(b_1,b_{r+1}) = \frac{1}{n}\left(\sum_{j=0}^{k_r} \left(\zeta_{u_0}^{\lambda}(u_1)\cdots\zeta_{u_m}^{\lambda}(u_0)\right)^j\right)\\
\cdot \frac{\zeta_{u_0}^{\lambda}(u_1)\cdots\zeta_{u_{i-1}}^{\lambda}(u_i)\zeta_{u_i}^{\lambda}(v_{i+1})\cdots\zeta_{v_{r-1}}^{\lambda}(v_r)|\Im \zeta_{v_r}^{\lambda}(v_{r+1})|^{1/2}}{|\Im \zeta_{u_0}^{\lambda}(u_1)|^{1/2}} \,,
\end{multline}
where $b_{r+1}=(v_r,v_{r+1})$ and $(u_i,v_{i+1},\dots,v_{r+1})$ is the unique path from $u_i$ to $v_{r+1}$.

To handle the above expressions, we must first control the sums over cycles.

\begin{prp}\label{prp:zetacyc}
Let $G$ be a connected graph of minimal degree $\ge 2$.

Suppose $G$ has a cycle $(u_0,\dots,u_m,u_0)$ which does not cover all vertices of $G$. Then $|\zeta_{u_0}^{\lambda}(u_1)\cdots\zeta_{u_{m-1}}^{\lambda}(u_m)\zeta_{u_m}^{\lambda}(u_0)| \le (1-\frac{z_{\lambda}^2}{4})$ for any $\lambda$ in the bulk spectrum $\cup_r J_r$.
\end{prp}

The proposition is also valid for all $\lambda\in \R$ if we replace the bound by $\le 1$.

\begin{proof}
As the cycle does not cover $G$, we may assume $\deg(u_0)\ge 3$. Then besides the two neighbors $u_1,u_m$, we know that $u_0$ has at least one neighbor $w_0$ outside the cycle.

Using \eqref{e:recurpath}, $\sum_{(v_2;v_{m+2})} |\zeta_{u_0}^{\lambda}(u_1)\zeta_{u_1}^{\lambda}(v_2)\cdots\zeta_{v_m}^{\lambda}(v_{m+1})|^2|\Im \zeta_{v_{m+1}}^{\lambda}(v_{m+2})| = |\Im \zeta_{u_0}^{\lambda}(u_1)|$, where the sum is over all non-backtracking paths $(v_2;v_{m+2})$ leaving the oriented edge $(u_0,u_1)$. There are at least two paths~: $(u_2,u_3,\dots,u_m,u_0,u_1)$ and $(u_2,u_3,\dots,u_m,u_0,w_0)$. Hence,
\[
|\zeta_{u_0}^{\lambda}(u_1)\cdots\zeta_{u_m}^{\lambda}(u_0)|^2\left(|\Im \zeta_{u_0}^{\lambda}(u_1)|+|\Im\zeta_{u_0}^{\lambda}(w_0)|\right) \le |\Im \zeta_{u_0}^{\lambda}(u_1)| \,,
\]
so $|\zeta_{u_0}^{\lambda}(u_1)\cdots\zeta_{u_m}^{\lambda}(u_0)|^2 \le \frac{|\Im \zeta_{u_0}^{\lambda}(u_1)|}{|\Im \zeta_{u_0}^{\lambda}(u_1)|+|\Im\zeta_{u_0}^{\lambda}(w_0)|}$. Since $\frac{|\Im \zeta_{u_0}^{\lambda}(w_0)|}{|\Im \zeta_{u_0}^{\lambda}(u_1)|+|\Im\zeta_{u_0}^{\lambda}(w_0)|} \ge \frac{z_{\lambda}}{2 z_{\lambda}^{-1}}$ by \eqref{e:zboun1}, we get $|\zeta_{u_0}^{\lambda}(u_1)\cdots\zeta_{u_m}^{\lambda}(u_0)|^2 \le 1- \frac{z_{\lambda}^2}{2} \le (1-\frac{z_{\lambda}^2}{4})^2$.
\end{proof}

From this, we get $|\cM_{n,\lambda}(b_1,b_1)|^2\le \frac{16 z_{\lambda}^{-4}}{n^2}$ and
\begin{equation}\label{e:keroncyc}
|\cM_{n,\lambda}(b_1,e_{r+1})|^2 \le \frac{16 z_{\lambda}^{-4}}{n^2} \cdot \frac{|\zeta_{u_0}^{\lambda}(u_1)\cdots\zeta_{u_{r-1}}^{\lambda}(u_r)|^2|\Im\zeta_{u_r}^{\lambda}(u_{r+1})|}{|\Im \zeta_{u_0}^{\lambda}(u_1)|} \le \frac{16 z_{\lambda}^{-4}}{n^2} \,,
\end{equation}
where the last inequality is due to \eqref{e:recurpath}, since $(u_2;u_{r+1})$ is just one of the paths $(x_2;x_{r+1})$.

For general points $b_{r+1}$ (in and out of $C$), let $\{u_{j_1},\dots,u_{j_s}\}$ be the vertices in $C$ of degree $\ge 3$. Using \eqref{e:outbr} and Proposition~\ref{prp:zetacyc} again, we have for $C_{\lambda}=16z_{\lambda}^{-4}$,
\begin{multline*}
\sum_{b'\in B} |\cM_{n,\lambda}(b_1,b')|^2 = \sum_{r=1}^n\left(\sum_{b_{r+1}\in \cB^rb_1 \cap C} + \sum_{b_{r+1}\in \cB^rb_1\cap C^c}\right) |\cM_{n,\lambda}(b_1,b_{r+1})|^2 \\
\le \frac{C_{\lambda}}{n} + \frac{C_{\lambda}}{n^2}\sum_{i=1}^s|\zeta_{u_0}^{\lambda}(u_1)\cdots\zeta_{u_{j_i-1}}^{\lambda}(u_{j_i})|^2\sum_{r=j_i}^n \sum_{(v_{j_i+1};v_{r+1})}\frac{|\zeta_{u_{j_i}}^{\lambda}(v_{j_i+1})\cdots\zeta_{v_{r-1}}^{\lambda}(v_r)|^2|\Im \zeta_{v_r}^{\lambda}(v_{r+1})|}{|\Im \zeta_{u_0}^{\lambda}(u_1)|}\,,
\end{multline*}
where we used \eqref{e:keroncyc} and the fact that $\cB^rb_1 \cap C$ is reduced to one edge (namely $e_{t+1}$, where $t=r$ $\mod m$). The last sum is over all paths outgoing from $(u_{j_i-1},u_{j_i})$ and leaving $C$. In particular, from \eqref{e:recurpath},
\[
\sum_{b'\in B} |\cM_{n,\lambda}(b_1,b')|^2 \le \frac{C_{\lambda}}{n} + \frac{C_{\lambda}}{n}\sum_{i=1}^s\frac{|\zeta_{u_0}^{\lambda}(u_1)\cdots\zeta_{u_{j_i-2}}^{\lambda}(u_{j_i-1})|^2|\Im \zeta_{u_{j_i-1}}^{\lambda}(u_{j_i})|}{|\Im \zeta_{u_0}^{\lambda}(u_1)|} \,.
\]
Finally, using \eqref{e:recurpath} again, if $(u_0,u_1,x_2;x_{m+1})$ are the $m$-paths outgoing from $(u_0,u_1)$, then $\sum_{(x_2;x_{m+1})} |\zeta_{u_0}^{\lambda}(u_1)\zeta_{u_1}^{\lambda}(x_2)\cdots\zeta_{x_{m-1}}^{\lambda}(x_m)|^2|\Im \zeta_{x_m}^{\lambda}(x_{m+1})| = |\Im\zeta_{u_0}^{\lambda}(u_1)|$. In particular,
\begin{multline*}
\sum_{i=1}^s |\zeta_{u_0}^{\lambda}(u_1)\cdots\zeta_{u_{j_i-1}}^{\lambda}(u_{j_i})|^2\sum_{(x_{j_i+1};x_{m+1})} |\zeta_{u_{j_i}}^{\lambda}(x_{j_i+1})\cdots\zeta_{x_{m-1}}^{\lambda}(x_m)|^2|\Im\zeta_{x_m}^{\lambda}(x_{m+1})| \\
\le |\Im \zeta_{u_0}^{\lambda}(u_1)| \,,
\end{multline*}
as this is the same sum restricted to paths leaving the cycle. Applying \eqref{e:recurpath} to the left-hand side, we finally obtain
\[
\sum_{b'\in B} |\cM_{n,\lambda}(b_1,b')|^2 \le \frac{32z_{\lambda}^{-4}}{n} \,.
\]

Back to \eqref{e:cauchy-schwarz}, note that since $\|\psi_{\lambda}\|=1$, we have $\|f_{\lambda}\| \le \sqrt{D}(1+z_{\lambda}^{-1})$ by \eqref{e:zboun1} and $|\Im \zeta^{\lambda}|\ge z_{\lambda}$. Here $D$ is the maximal degree of $G$. Finally returning to \eqref{e:rebulk}, estimating the $g_{\lambda}$ similarly, we get
\[
\left\|\psi_{\lambda}\right\|_{\infty} \le \frac{1}{2z_{\lambda}} \cdot \frac{\sqrt{32}z_{\lambda}^{-2}}{\sqrt{n}} \cdot (2\sqrt{D} (1+z_{\lambda}^{-1})) \le \frac{2\sqrt{32D} z_{\lambda}^{-4}}{\sqrt{n}} \,,
\]
where we used that $z_{\lambda}\le 1$, as follows from $z_{\lambda}\le |\zeta^{\lambda}|\le z_{\lambda}^{-1}$. The bound $\sqrt{32D}=4\sqrt{2D}\le 4D$ completes the proof.

\subsection{Outside the spectrum} \label{sec:outsidesup}

To deal with energies outside the spectrum, we use a Combes-Thomas estimate. This tool is a common ingredient in the theory of random Schr\"odinger operators to control the Green kernel.

Intuitively, if $\lambda\notin \sigma(H)$, then $\|G^{\lambda} \delta_x\|^2 = \sum_y |G^{\lambda}(y,x)|^2<\infty$, implying that $|G^{\lambda}(y,x)|$ must decay with $d(x,y)$. The Combes-Thomas estimate says that if $H$ is a Schr\"odinger operator, then the decay must be exponential.

\begin{lem}\label{lem:CT}
Let $\mathbb{G}$ be a countable graph with degree bounded by $D$ and $H = \cA+W$ a self-adjoint Schr\"odinger operator. Suppose $\lambda\notin\sigma(H)$ and let $\delta_{\lambda}=\dist(\lambda,\sigma(H))$. Then for any $x\in \mathbb{G}$,
\[
\sum_{y:\,d_{\mathbb{G}}(x,y)=n} |G^{\lambda}(y,x)|^2 \le \frac{4}{\delta_{\lambda}^2} \cdot \frac{1}{(1+\frac{\delta_{\lambda}}{2D})^{2n}} \,.
\]
\end{lem}
The estimate holds more generally if we replace $\cA$ by $(\cA_p f)(x) = \sum_{y\sim x} p_y(x) f(y)$, $p_y(x)=p_x(y)$, in which case $D$ should be replaced by $\max_x \sum_{y\sim x} p_x(y)$.
\begin{proof}
The proof is contained in \cite[Theorem 10.5]{AW15} or \cite[Theorem 11.2]{Ki08}.

Fix a cutoff $R>0$ and consider the multiplication operator $M = e^{\mu \min(d(\cdot,x),R)}$, where $\mu>0$ will be determined later. Then $M \delta_u = c_{u,x,R} \delta_u$, where $c_{u,x,R} = e^{\mu \min(d(u,x),R)}$. In particular, if $d(y,x)=n \le R$, we get
\[
e^{\mu n} G^{\lambda}(y,x) = \langle \delta_y, M(H-\lambda)^{-1}M^{-1} \delta_x\rangle = \langle \delta_y, (MHM^{-1}-\lambda)^{-1}\delta_x\rangle \,.
\]
Hence, if $B:=MHM^{-1}-H$,
\begin{multline*}
\sum_{y,d(y,x)=n} |G^{\lambda}(y,x)|^2 = e^{-2\mu n} \sum_{y,d(y,x)=n} \left|[(MHM^{-1}-\lambda)^{-1}\delta_x](y)\right|^2 \\
\le e^{-2\mu n} \|(MHM^{-1}-\lambda)^{-1}\delta_x\|^2 \le e^{-2\mu n} \|(H+B-\lambda)^{-1}\|^2 \,.
\end{multline*}
Now $B(u,u') = c_{u,x,R}c_{u',x,R}^{-1}-1$ if $u\sim u'$ and $B(u,u')=0$ otherwise. In particular, $|B(u,u')| \le e^{\mu}-1$, $\sum_{u'} |B(u,u')| \le D(e^{\mu} -1)$ and $\sum_u |B(u,u')| \le D(e^{\mu}-1)$. If follows that $\|B\| \le D(e^{\mu}-1)$. Recall $\delta_{\lambda} = \dist(\lambda,\sigma(H))$. As is well-known, if $\|B\|<\delta_{\lambda}$, then $\|(H+B-\lambda)^{-1}\|\le \frac{1}{\delta_{\lambda}-\|B\|}$. Choosing $\mu = \ln(1+\frac{\delta_{\lambda}}{2D})$ proves the claim.
\end{proof}

We may now prove part (2) of Theorem~\ref{thm:main}. Let $\lambda\notin \sigma(H_{\cT})$. We use Proposition~\ref{prp:retour} with $k=r$; focus on the first sum. For $n\le \ell_G$, we have
\begin{align*}
\left|\sum_{(x_2;x_{r+1})} G^{\lambda}(\tilde{x}_0,\tilde{x}_{r+1})\psi_{\lambda}(x_r)\right|^2 &\le D\sum_{(x_2;x_{r+1})}|G^{\lambda}(\tilde{x}_0,\tilde{x}_{r+1})|^2\sum_{(x_2;x_r)}|\psi_{\lambda}(x_r)|^2 \\
& \le 2D\sum_{(x_2;x_{r+1})}|G^{\lambda}(\tilde{x}_0,\tilde{x}_{r+1})|^2\sum_{y:d_G(y,x_0)\le r}|\psi_{\lambda}(y)|^2 \,,
\end{align*}
where in the first inequality we used Cauchy-Schwarz and the fact that $\#\{x_{r+1}\} =|\cN_{x_r}|-1 < D$. In the second inequality, the factor $2$ comes from the second sum involving $\psi_{\lambda}$~: as mentioned in Remark~\ref{rem:tangle}, there can be at most two paths of a fixed length $r$ between $x_0$ and $y\in B_G(x_0,n)$. The first sum is over a subset of $\{v:d_{\cT}(\tilde{x}_0,v)=r+1\}$, so by Lemma~\ref{lem:CT}, we get
\[
\left|\sum_{(x_2;x_{r+1})} G^{\lambda}(\tilde{x}_0,\tilde{x}_{r+1})\psi_{\lambda}(x_r)\right|^2 \le \frac{8D}{\delta_{\lambda}^2} \cdot \frac{1}{(1+\frac{\delta_{\lambda}}{2D})^{2r}} \|\psi_{\lambda}\|_2^2 \,.
\]
The other sums in Proposition~\ref{prp:retour} are bounded similarly. We conclude using $8\sqrt{2D}\le 8D$.

\subsection{Proof of the corollary}

Finally, since $\|\psi_{\lambda}\|_2^2=1$, we deduce that
\[
\|\psi_{\lambda}\|_p = \left(\sum_{x\in V} |\psi_{\lambda}(x)|^p\right)^{1/p} \le \|\psi_{\lambda}\|_{\infty}^{\frac{p-2}{p}} \left(\sum_{x\in V}|\psi_{\lambda}(x)|^2\right)^{1/p} = \|\psi_{\lambda}\|_{\infty}^{\frac{p-2}{p}} \,.
\]
For the second part of the corollary, simply observe that
\[
\varepsilon \le \|\chi_{\Lambda}\psi_{\lambda}\|_2^2 = \sum_{x\in \Lambda} |\psi_{\lambda}(x)|^2 \le \|\psi_{\lambda}\|_{\infty}^2 \cdot |\Lambda| \,.
\]
This completes the proof of Theorem~\ref{thm:main} and Corollary~\ref{cor:bl}.

\section{Case of Schr\"odinger cycles} \label{sec:cycles}

In this section, we consider the case of Schr\"odinger operators on an $N$-cycle $C_N = \{0,\dots,N-1\}$. We denote the potentials by $W_j = W(j)$.

Note that Theorem~\ref{thm:specone} does not apply to this model as \textbf{(C1)} does not hold. Fortunately, the analog of Theorem~\ref{thm:specone} already exists in the literature. In fact, the universal cover of $(C_N,W)$ is just $(\Z,W)$ and $H_{\Z}$ is just a periodic Schr\"odinger operator on $\Z$, where $W_{j+kN} = W_j$.

\begin{thm}[cf. \cite{Si11}] \label{thm:specone2}
Let $(\Z,W)$ be the universal cover of $(C_N,W)$.

\begin{enumerate}[\rm (i)]
\item The spectrum of $H_{\Z}$ is purely absolutely continuous, and consists of at most $N$ bands~: $\sigma(H_{\Z}) = \cup_{r=1}^N I_r$.
\item If $W$ is $m$-periodic on $C_N$, say $N=mN'$ with $W_{j+m}=W_j$, the spectrum of $H_{\Z}$ has at most $m$ bands. Moreover, $G_{\Z}^{\lambda+\ii0}(j+km,j+km) = G_{\Z}^{\lambda+\ii0}(j,j)$ and $\zeta_{j+km}^{\lambda+\ii 0}(j+ km \pm 1) = \zeta_j^{\lambda+\ii0}(j\pm 1)$.
\item The limits $\zeta_j^{\lambda+\ii0}(j\pm 1)$ and $G^{\lambda+\ii 0}_{\Z}(j,k)$ exists for any $j,k\in\Z$ if $\lambda$ is in the interior $\mathring{I}_r$. At least one Green function $G^{\lambda+\ii\eta}(j,j)$ blows up as $\eta \downarrow 0$ at the endpoints of $I_r$.
\item For $\lambda\in\mathring{I}_r$, the Green functions are pure imaginary~: $\Re G^{\lambda+\ii0}(j,j) = 0$.
\item In general, $G^{\lambda+\ii 0}(j,j)$ only vanishes in the gaps between $I_r$. More precisely, $G^{\lambda+\ii 0}(j,j)$ has exactly one zero in each gap (this zero can occur at an endpoint of $I_r$).
\item If $\lambda\in \mathring{I}_r$, then $|\Im \zeta_j^{\lambda+\ii0}(j\pm 1) | > 0$ and $\Im G^{\lambda+\ii0}(j,j)>0$ for all $j$.
\end{enumerate}
\end{thm}
\begin{proof}
We simply indicate the references. For (i), see \cite[Theorem 5.3.7]{Si11}. If $W$ is $m$-periodic, $H_{\Z}$ is also $m$-periodic, so we have at most $m$ bands. The periodicity of the Green functions simply follows from the isomorphisms of the rooted graphs $(\Z,W,j)$ and $(\Z,W,j+km)$ for any root $j$. Alternatively, one sees this from \cite[eq. (5.4.72)]{Si11}.

For (iii)--(v), see Corollary 5.4.6, Theorem 5.4.15 and Theorem 5.4.17 in \cite{Si11}. Note that with our notations, $\zeta_j^{\gamma}(j+1) = \frac{u_{j+1}^+(\gamma)}{u_j^+(\gamma)}$ and $\zeta_{j+1}^{\gamma}(j) = \frac{u_j^-(\gamma)}{u_{j+1}^-(\gamma)}$ correspond to the $m$-functions in \cite{Si11}, and $G^{\lambda+\ii0}(j,j) = \frac{|u_j^+(\lambda)|^2}{2\ii \Im u_1^+(\lambda)}$. Within $\mathring{I}_k$, the spectral measure is $\dd \nu(\lambda) = \frac{1}{\pi N} |\theta'(\lambda)|\,\dd\lambda = \frac{1}{N\pi}\sum_{r=0}^{N-1} \Im G^{\lambda+\ii0}(r,r)$, see \cite[eq. (5.3.34), (5.10.14)]{Si11}, with $|\theta'(\lambda)|=|\frac{1}{e_k'(\theta_{\lambda})}|>0$ for all $\lambda\in\mathring{I}_k$ (see \cite[Theorem 5.3.4]{Si11}). The function $e_k$ is analytic on $\mathring{I}_k$. All this implies all Green functions $G^{\lambda+\ii0}(j,j)$ exist on $\mathring{I}_k$, and at least one of them has a positive imaginary part. In particular, all $u_j^+(\lambda)$ exist on $\mathring{I}_k$, we have $|\Im u_1^+(\lambda)|>0$ and $u_0^{\pm}(\lambda)=1$ by definition. This shows $\zeta_j^{\lambda+\ii 0}(j+1)$ exists for $j=0,1$ and $|\Im\zeta_0^{\lambda+\ii0}(1)|>0$. The recursive formula \eqref{e:rec} here is just $\frac{1}{\zeta_{j-1}^{\gamma}(j)} = \gamma-W_j-\zeta_j^{\gamma}(j+1)$. It follows that $|\Im \zeta_1^{\lambda+\ii0}(2)|>0$. By induction, if $\zeta_{j-1}^{\lambda+\ii 0}(j)$ exists and $\Im\zeta_{j-1}^{\lambda+\ii 0}(j)$ is nonzero, then in particular $\zeta_{j-1}^{\lambda+\ii0}(j)\neq 0$, so $\zeta_j^{\lambda+\ii0}(j+1)$ exists and $|\Im \zeta_j^{\lambda+\ii0}(j+1)|>0$ for all $j$. Using \eqref{e:rec}, this shows all $G^{\lambda+\ii 0}(j,j)$ have a positive imaginary part. Finally, by \cite[Theorem 5.4.15]{Si11}, we have $u_j^-(\lambda) = \overline{u_j^+(\lambda)}$ on $\mathring{I}_r$, so we deduce the same results for $\zeta_{j+1}^{\lambda+\ii0}(j)$. This completes the proof. 
\end{proof}

We may now proceed with the eigenfunction estimates. Here, the sums $\sum_{(x_2;x_{r+1})}$ in \eqref{e:outrep} reduce to a single path. We take $k=r$ and apply \eqref{e:outrep} for each $r\le n$ to get
\begin{multline*}
\psi_{\lambda}(j) = \frac{1}{n}\sum_{r=1}^n \left[G^{\lambda}(j,j+r+1)\psi_{\lambda}(j+r) - G^{\lambda}(j,j+r)\psi_{\lambda}(j+r+1) \right] \\
+ \frac{1}{n}\sum_{r=1}^n\left[G^{\lambda}(j,j-r)\psi_{\lambda}(j-r+1) - G^{\lambda}(j,j-r+1)\psi_{\lambda}(j-r)\right] .
\end{multline*}
In fact, here the $\cB$ from \eqref{e:nonbackoper} is just a shift~: $(\cB f)(j,j+1) = f(j+1,j+2)$.

We take $n=N$ and use Cauchy-Schwarz~:
\[
\left|\frac{1}{N}\sum_{r=1}^N G^{\lambda}(j,j+r+1)\psi_{\lambda}(j+r)\right|^2 \le \frac{1}{N}\sum_{r=1}^N|G^{\lambda}(j,j+r+1)|^2|\psi_{\lambda}(j+r)|^2 \,.
\]

For $\lambda\notin\sigma(H_{\Z})$, we just bound $|G^{\lambda}(j,j+r+1)| \le \frac{1}{\delta_{\lambda}}$, then the right-hand side is bounded by $\frac{\delta_{\lambda}^{-2}}{N} \|\psi_{\lambda}\|_2^2$, proving the claim. For $\lambda\in \mathring{I}_r$, using \eqref{e:rec}, we know $\frac{|\Im\zeta_{j-1}^{\lambda}(j)|}{|\zeta_{j-1}^{\lambda}(j)|^2}=|\Im \zeta_j^{\lambda}(j+1)|$, so by \eqref{e:greenmulti}, we have
\[
|G^{\lambda}(j,j+r+1)|^2 = |G^{\lambda}(j,j)\zeta_j^{\lambda}(j+1)\cdots\zeta_{j+r}^{\lambda}(j+r+1)|^2=\frac{|G^{\lambda}(j,j)|^2\cdot |\Im\zeta_j^{\lambda}(j+1)|}{|\Im \zeta_{j+r+1}^{\lambda}(j+r+2)|} \,.
\]
It follows by Theorem~\ref{thm:specone2} that $|G^{\lambda}(j,j+r+1)|^2\le C_{\lambda}$ and thus $\|\psi_{\lambda}\|_{\infty}^2 \le \frac{4C_{\lambda}}{N}\|\psi_{\lambda}\|_2^2$. Note that for $\lambda\in\mathring{I}_r$, we are assuming the potential is $m$-periodic, so all Green functions are those of $m$-periodic Schr\"odinger operators, and the constant $C_{\lambda}$ may only depend on $m$ (but not on $N$).

\section{The proof for p-norms}

The aim of this section is to prove Theorem~\ref{thm:psupp}.

\subsection{A TT* argument}
Recall the operator
\[
\cM_{n,\lambda} = \frac{1}{n}\sum_{r=1}^n\frac{1}{|\Im \zeta^{\lambda}|^{1/2}}\left(\zeta^{\lambda}\cB\right)^r|\Im\zeta^{\lambda}|^{1/2} \,.
\]
For $\lambda$ in the bulk spectrum, we deduce from \eqref{e:rebulk} that
\begin{equation}\label{e:pnormpsif}
\|\psi_{\lambda}\|_p \le \left\|\frac{1}{|\Im\zeta^{\lambda}|^{1/2}}\cM_{n,\lambda}\frac{f_{\lambda}}{|\Im \zeta^{\lambda}|^{1/2}}\right\|_p + \left\|\frac{1}{|\Im\zeta^{\lambda}|^{1/2}}\overline{\cM_{n,\lambda}}\frac{g_{\lambda}}{|\Im \zeta^{\lambda}|^{1/2}}\right\|_p\,,
\end{equation}
where $\zeta^{\lambda}(x_0,x_1):=\zeta_{x_0}^{\lambda}(x_1)$.

We study the first term; the other one is similar. Recalling \eqref{e:zlam}, we have
\begin{equation}\label{e:flambdap}
\left\|\frac{1}{|\Im\zeta^{\lambda}|^{1/2}}\cM_{n,\lambda}\frac{f_{\lambda}}{|\Im \zeta^{\lambda}|^{1/2}}\right\|_p \le z_{\lambda}^{-1}\|\cM_{n,\lambda}\|_{2\to p} \cdot \|f_{\lambda}\|_2 \,.
\end{equation}
On the other hand, we can use a so-called $TT^{\ast}$ argument, namely
\begin{equation}\label{e:ttstarn}
\|\cM_{n,\lambda}\|_{2\to p}^2 = \|\cM_{n,\lambda}^{\ast}\|_{p'\to 2}^2 = \| \cM_{n,\lambda}\cM_{n,\lambda}^{\ast} \|_{p'\to p} \,.
\end{equation}
So it suffices to estimate this operator norm to obtain a bound on the $p$-norm.

We start by calculating
\begin{equation}\label{e:ttstar}
\cM_{n,\lambda}\cM_{n,\lambda}^{\ast} = \frac{1}{n^2}\sum_{r,r'=1}^n\frac{1}{|\Im\zeta^{\lambda}|^{1/2}}\left(\zeta^{\lambda}\cB\right)^r|\Im\zeta^{\lambda}|\left(\cB^{\ast}\overline{\zeta^{\lambda}}\right)^{r'}\frac{1}{|\Im\zeta^{\lambda}|^{1/2}} \,.
\end{equation}
We have
\[
\left[\left(\zeta^{\lambda}\cB\right)^rf\right](x_0,x_1) = \sum_{(x_2;x_{r+1})} \zeta_{x_0}^{\lambda}(x_1)\cdots\zeta_{x_{r-1}}^{\lambda}(x_r) f(x_r,x_{r+1}) \,,
\]
where the sum runs over all paths $(x_2;x_{r+1})$ such that $x_2\in \cN_{x_1}\setminus \{x_0\}$. Similarly,
\[
\left[\left(\cB^{\ast}\overline{\zeta^{\lambda}}\right)^{r'}g\right](x_r,x_{r+1}) = \sum_{(y_0;y_{r-1})}\overline{\zeta_{y_0}^{\lambda}(y_1)\cdots\zeta_{y_{r-1}}^{\lambda}(x_r)} g(y_0,y_1) \,,
\]
where the sum runs over all paths $(y_0;y_{r-1})$ such that $y_{r-1}\in\cN_{x_r}\setminus\{x_{r+1}\}$. Hence,
\begin{multline*}
\left[\frac{1}{|\Im\zeta^{\lambda}|^{1/2}}\left(\zeta^{\lambda}\cB\right)^r|\Im\zeta^{\lambda}|\left(\cB^{\ast}\overline{\zeta^{\lambda}}\right)^r\frac{1}{|\Im\zeta^{\lambda}|^{1/2}}f\right](x_0,x_1) \\
= \sum_{(x_2;x_{r+1})}\sum_{(y_0;y_{r-1})} \frac{\zeta_{x_0}^{\lambda}(x_1)\cdots\zeta_{x_{r-1}}^{\lambda}(x_r)|\Im\zeta_{x_r}^{\lambda}(x_{r+1})|\,\overline{\zeta_{y_0}^{\lambda}(y_1)\cdots \zeta_{y_{r-1}}^{\lambda}(x_r)}}{|\Im\zeta_{x_0}^{\lambda}(x_1)|^{1/2}|\Im\zeta_{y_0}^{\lambda}(y_1)|^{1/2}} f(y_0,y_1) \,.
\end{multline*}

This double sum has complicated combinatorics~: we first sum on $r$-paths which are outgoing from $(x_0,x_1)$, then for each edge $e_r=(x_r,x_{r+1})$ in this $r$-sphere, we sum on $r$-paths which are in the ``past'' of $e_r$, i.e. outgoing from $(x_{r+1},x_r)$. Our aim now is to simplify this expression.

Fix $(x_2;x_{r+1})$. We write $\{(y_0;y_{r-1}):y_{r-1}\in \cN_{x_r}\setminus \{x_{r+1}\}\}$ as
\begin{multline*}
\left\{(x_0;x_{r-1})\right\} \cup \left\{(y_0,x_1,\dots,x_{r-1}):y_0\in\cN_{x_1}\setminus\{x_2\},y_0\neq x_0\right\} \\
\cup \cdots \cup\left\{(y_0;y_{r-1}):y_{r-1}\in \cN_{x_r}\setminus\{x_{r+1}\},y_{r-1}\neq x_{r-1}\right\} \,.
\end{multline*}
For $(y_0;y_{r-1}) = (x_0;x_{r-1})$, we get the term
\[
\left[\sum_{(x_2;x_{r+1})} \frac{|\zeta^{\lambda}_{x_0}(x_1)\cdots \zeta^{\lambda}_{x_{r-1}}(x_r)|^2|\Im\zeta_{x_r}^{\lambda}(x_{r+1})|}{|\Im\zeta_{x_0}^{\lambda}(x_1)|}\right]\cdot f(x_0,x_1) = f(x_0,x_1)\,.
\]
For the terms $(y_0,\dots, y_j,x_{j+1},\dots, x_{r-1})$, the restriction $y_j \in \cN_{x_{j+1}} \setminus \{x_{j+2}\}$, $y_j \neq x_j$ is equivalent to $x_{j+2}\in \cN_{x_{j+1}} \setminus \{x_j,y_j\}$, yielding for each $j=0,\dots,r-1$ the term
\begin{multline*}
\sum_{(x_2;x_{j+1})}\sum_{(y_0;y_j)'} \bigg[\sum_{(x_{j+2};x_{r+1}),x_{j+2}\in \cN_{x_{j+1}}\setminus\{x_j,y_j\}} \frac{\zeta^{\lambda}_{x_0}(x_1)\cdots \zeta^{\lambda}_{x_j}(x_{j+1}) \cdot \overline{\zeta^{\lambda}_{y_0}(y_1)\cdots \zeta^{\lambda}_{y_j}(x_{j+1})}}{|\Im\zeta_{x_0}^{\lambda}(x_1)|^{1/2}|\Im\zeta_{y_0}^{\lambda}(y_1)|^{1/2}} \\\cdot |\zeta^{\lambda}_{x_{j+1}}(x_{j+2})\cdots \zeta^{\lambda}_{x_{r-1}}(x_r)|^2|\Im\zeta_{x_r}^{\lambda}(x_{r+1})|\bigg]\cdot f(y_0,y_1)\\
= \sum_{(x_2;x_{j+1})}\sum_{(y_0;y_j)'}  \frac{\zeta_{x_0}^{\lambda}(x_1)\cdots \zeta_{x_j}^{\lambda}(x_{j+1}) \cdot \overline{\zeta^{\lambda}_{y_0}(y_1)\cdots \zeta^{\lambda}_{y_j}(x_{j+1})}}{|\Im \zeta_{x_0}^{\lambda}(x_1)\Im \zeta_{y_0}^{\lambda}(y_1)|^{1/2}} \\\cdot \left(\frac{|\Im \zeta_{x_j}^{\lambda}(x_{j+1})|}{|\zeta^{\lambda}_{x_j}(x_{j+1})|^2} - |\Im \zeta_{x_{j+1}}^{\lambda}(y_j)|\right) f(y_0,y_1)\,,
\end{multline*}
where $\sum_{(y_0;y_j)'}$ sums over all paths $(y_0;y_j)$ with $y_j\in \cN_{x_{j+1}} \setminus \{x_j\}$.

In the last expression, $\sum_{(x_2;x_{j+1})}\sum_{(y_0;y_j)'}$ can be rewritten as $\sum_{(x_2;x_{2j+2})}$, with $y_j$ re-labeled as $x_{j+2}$, $y_{j-1}$ as $x_{j+3}$ and so on. This gives
\begin{multline*}
\sum_{(x_2;x_{2j+2})} \frac{\zeta_{x_0}^{\lambda}(x_1)\cdots \zeta_{x_j}^{\lambda}(x_{j+1}) \cdot \overline{\zeta^{\lambda}_{x_{2j+2}}(x_{2j+1})\cdots \zeta^{\lambda}_{x_{j+2}}(x_{j+1})}}{|\Im \zeta_{x_0}^{\lambda}(x_1)\Im \zeta_{x_{2j+2}}^{\lambda}(x_{2j+1})|^{1/2}} \\\cdot \left(\frac{|\Im \zeta_{x_j}^{\lambda}(x_{j+1})|}{|\zeta^{\lambda}_{x_j}(x_{j+1})|^2} - |\Im \zeta_{x_{j+1}}^{\lambda}(x_{j+2})|\right) f(x_{2j+2},x_{2j+1}) \,.
\end{multline*}

Now note that $[\frac{1}{|\Im \zeta|^{1/2}} (\zeta \cB)^j \frac{|\Im \zeta|}{\overline{\zeta}} g](x_0,x_1) = \sum_{(x_2;x_{j+1})} \frac{\zeta_{x_0}(x_1)\cdots \zeta_{x_j}(x_{j+1})}{|\Im \zeta_{x_0}(x_1)|^{1/2}} \frac{|\Im \zeta_{x_j}(x_{j+1})|}{|\zeta_{x_j}(x_{j+1})|^2}\cdot g(x_j,x_{j+1})$. So the above is, in operator form,
\begin{multline*}
\left(\frac{1}{|\Im \zeta^{\lambda}|^{1/2}} \left(\zeta^{\lambda}\cB\right)^j \frac{|\Im \zeta^{\lambda}|}{\overline{\zeta^{\lambda}}} \left(\cB \iota \overline{\zeta^{\lambda}}\right)^{j+1} \frac{1}{|\Im \iota \zeta^{\lambda}|^{1/2}} \iota f\right)(x_0,x_1) \\
- \left(\frac{1}{|\Im \zeta^{\lambda}|^{1/2}} \left(\zeta^{\lambda}\cB\right)^j \zeta^{\lambda} \left(\cB |\Im \zeta^{\lambda}| \iota \overline{\zeta^{\lambda}}\right)\left(\cB \iota \overline{\zeta^{\lambda}}\right)^j \frac{1}{|\Im \iota \zeta^{\lambda}|^{1/2}} \iota f\right)(x_0,x_1) \,,
\end{multline*}
where $\iota$ is edge-reversal. In other words, if
\begin{equation}\label{e:bzeta}
\cB_{\zeta^{\lambda}} = \frac{|\Im \zeta^{\lambda}|}{\overline{\zeta^{\lambda}}} \cB \iota \overline{\zeta^{\lambda}} - \zeta^{\lambda} \left(\cB |\Im \zeta^{\lambda}| \iota \overline{\zeta^{\lambda}}\right) \,,
\end{equation}
then
\begin{multline*}
\frac{1}{|\Im\zeta^{\lambda}|^{1/2}}\left(\zeta^{\lambda}\cB\right)^r|\Im\zeta^{\lambda}|\left(\cB^{\ast}\overline{\zeta^{\lambda}}\right)^r\frac{1}{|\Im\zeta^{\lambda}|^{1/2}} \\
= I +\sum_{j=0}^{r-1} \frac{1}{|\Im \zeta^{\lambda}|^{1/2}} \left(\zeta^{\lambda}\cB\right)^j \cB_{\zeta^{\lambda}} \left(\cB \iota \overline{\zeta^{\lambda}} \right)^j \frac{1}{|\Im \iota \zeta^{\lambda}|^{1/2}} \iota \, .
\end{multline*}
It follows that for $r\ge r'$,
\begin{multline}\label{e:cool2}
\frac{1}{|\Im\zeta^{\lambda}|^{1/2}}\left(\zeta^{\lambda}\cB\right)^r|\Im\zeta^{\lambda}|\left(\cB^{\ast}\overline{\zeta^{\lambda}}\right)^{r'}\frac{1}{|\Im\zeta^{\lambda}|^{1/2}} \\
= \frac{1}{|\Im \zeta^{\lambda}|^{1/2}} \left(\zeta^{\lambda} \cB\right)^{r-r'} |\Im \zeta^{\lambda}|^{1/2} \\
+\sum_{j=0}^{r'-1} \frac{1}{|\Im \zeta^{\lambda}|^{1/2}} \left(\zeta^{\lambda}\cB\right)^{j+r-r'} \cB_{\zeta^{\lambda}} \left(\cB \iota \overline{\zeta^{\lambda}} \right)^j \frac{1}{|\Im \iota \zeta^{\lambda}|^{1/2}} \iota
\end{multline}
and for $r<r'$,
\begin{multline}\label{e:cool3}
\frac{1}{|\Im\zeta^{\lambda}|^{1/2}}\left(\zeta^{\lambda}\cB\right)^r|\Im\zeta^{\lambda}|\left(\cB^{\ast}\overline{\zeta^{\lambda}}\right)^{r'}\frac{1}{|\Im\zeta^{\lambda}|^{1/2}} \\
= |\Im \zeta^{\lambda}|^{1/2} \left(\cB^{\ast}\overline{\zeta^{\lambda}}\right)^{r'-r} \frac{1}{|\Im \zeta^{\lambda}|^{1/2}} \\
+\sum_{j=0}^{r-1} \frac{1}{|\Im \zeta^{\lambda}|^{1/2}} \left(\zeta^{\lambda}\cB\right)^j \cB_{\zeta^{\lambda}} \left(\cB \iota \overline{\zeta^{\lambda}} \right)^{j+r'-r} \frac{1}{|\Im \iota \zeta^{\lambda}|^{1/2}} \iota \, .
\end{multline}

In the last sum, we used that $\cB^{\ast} = \iota \cB \iota$ and $\iota^2=I$.

\subsection{Some Green function estimates}\label{e:greenp}

Let $(G,W)$ be a finite graph of minimal degree $\ge 3$ and consider a path $(x_0;x_{r+1})$ in $G$. Assume $\lambda$ is in the bulk spectrum of $(G,W)$. We know from \eqref{e:magicim} that for any $j\ge 1$,
\[
\sum_{x_{j+1}\in \cN_{x_j}\setminus \{x_{j-1}\}} \frac{|\zeta_{x_{j-1}}^{\lambda}(x_j)|^2\,|\Im\zeta_{x_j}^{\lambda}(x_{j+1})|}{|\Im \zeta_{x_{j-1}}^{\lambda}(x_j)|} = 1 \,.
\]
Moreover, each term in the sum is strictly positive. If $d(x_j) \ge 3$, this implies that each term (there are at least two) is strictly smaller than $1$. It follows that for any $s>1$,
\[
\sum_{x_{j+1}\in\cN_{x_j}\setminus \{x_{j-1}\}} \frac{|\zeta_{x_{j-1}}^{\lambda}(x_j)|^{2s}|\Im\zeta_{x_j}^{\lambda}(x_{j+1})|^s}{|\Im \zeta_{x_{j-1}}^{\lambda}(x_j)|^s} <1 \,.
\]
Taking the maximum of the left-hand side over the finite set of oriented edges of $G$, we obtain $Z_{s,\lambda}(G)<1$, by definition \eqref{e:zslam}. Hence,
\begin{multline*}
\sum_{(x_2;x_{r+1})} \frac{|\zeta_{x_0}^{\lambda}(x_1)\cdots\zeta_{x_{r-1}}^{\lambda}(x_r)|^{2s}|\Im\zeta_{x_r}^{\lambda}(x_{r+1})|^s}{|\Im\zeta_{x_0}^{\lambda}(x_1)|^s} \\
\le Z_{s,\lambda} \sum_{(x_2;x_r)} \frac{|\zeta_{x_0}^{\lambda}(x_1)\cdots\zeta_{x_{r-2}}^{\lambda}(x_{r-1})|^{2s}|\Im\zeta_{x_{r-1}}^{\lambda}(x_r)|^s}{|\Im\zeta_{x_0}^{\lambda}(x_1)|^s} \le Z_{s,\lambda}^r \,.
\end{multline*}

On the other hand, $\frac{|\Im\zeta_{x_r}^{\lambda}(x_{r+1})|^s}{|\Im\zeta_{x_0}^{\lambda}(x_1)|^s} \ge z_{\lambda}^{2s}$ by \eqref{e:zboun1}. We thus get for any $s>1$ and $(x_0,x_1)$,
\begin{equation}\label{e:greensno}
\sum_{(x_2;x_{r+1})} |\zeta_{x_0}^{\lambda}(x_1)\cdots\zeta_{x_{r-1}}^{\lambda}(x_r)|^{2s} \le z_{\lambda}^{-2s} Z_{s,\lambda}^r \,.
\end{equation}
Note that in the case of $(q+1)$-regular graphs with $W\equiv 0$, we have more precisely $\sum_{(x_2;x_{r+1})} |\zeta_{x_0}^{\lambda}(x_1)\cdots\zeta_{x_{r-1}}^{\lambda}(x_r)|^{2s}  = q^r q^{-rs} = q^{-(s-1)r}$.

It follows in particular that
\begin{equation}\label{e:inftyboun}
|\zeta_{x_0}^{\lambda}(x_1)\cdots\zeta_{x_{r-1}}^{\lambda}(x_r)|^2 \le z_{\lambda}^{-2} (Z_{s,\lambda}^{1/s})^r
\end{equation}
decays exponentially in $r$.

We also have the following variant of \eqref{e:greensno}~:
\begin{equation}\label{e:greensno2}
\sum_{(x_2;x_{r+1})} |\zeta_{x_1}^{\lambda}(x_0)\cdots\zeta_{x_r}^{\lambda}(x_{r-1})|^{2s} \le z_{\lambda}^{-6s} Z_{s,\lambda}^r \,.
\end{equation}
In fact, $|\zeta_{x_1}^{\lambda}(x_0)\cdots\zeta_{x_r}^{\lambda}(x_{r-1})|^{2s} = |\frac{G^{\lambda}(\tilde{x}_0,\tilde{x}_0)}{G^{\lambda}(\tilde{x}_r,\tilde{x}_r)}|^{2s}|\zeta_{x_0}^{\lambda}(x_1)\cdots\zeta_{x_{r-1}}^{\lambda}(x_r)|^{2s}$ by \eqref{e:zetainv} and we know $|\frac{G^{\lambda}(v,v)}{G^{\lambda}(w,w)}| \le z_{\lambda}^{-2}$ by \eqref{e:zboun2}.

\subsection{Proof for p-norms}
Now recall Young's inequality~: if $K$ is an operator on $\ell^2(B)$, $\frac{1}{q}=\frac{1}{s}+\frac{1}{r}-1$, $\|K(\cdot,b)\|_r \le A$ and $\|K(b,\cdot)\|_r \le A$ for all $b\in B$, then $\|Kf\|_q \le A \|f\|_s$. We use this with $q=p$ and $s=p'$, so that $r=\frac{p}{2}$. Our operator is $K = \cM_{n,\lambda}\cM_{n,\lambda}^{\ast}$.

Recall that
\[
\cM_{n,\lambda} = \frac{1}{n}\sum_{r=1}^n \cB_{r,\lambda}\,, \qquad \text{where} \quad  \cB_{m,\lambda} = \frac{1}{|\Im \zeta^{\lambda}|^{1/2}}(\zeta^{\lambda}\cB)^m|\Im \zeta^{\lambda}|^{1/2} \,,
\]
so that $\cM_{n,\lambda}\cM_{n,\lambda}^{\ast} = \frac{1}{n^2}\sum_{r,r'=1}^n \cB_{r,\lambda}\cB_{r',\lambda}^{\ast}$. We focus on $\sum_{r=1}^n\sum_{r'=1}^r$, the rest $\sum_{r=1}^n\sum_{r'=r+1}^n$ is bounded similarly.

Using \eqref{e:cool2}, if
\[
\cB_{m,j,\lambda} = \frac{1}{|\Im\zeta^{\lambda}|^{1/2}} (\zeta^{\lambda}\cB)^m\cB_{\zeta^{\lambda}}(\cB\iota\overline{ \zeta^{\lambda}})^j\frac{1}{|\Im\iota \zeta^{\lambda}|^{1/2}}\,,
\]
then for $r\ge r'$,
\[
 \cB_{r,\lambda}\cB_{r',\lambda}^{\ast} = \cB_{r-r',\lambda} + \sum_{j=0}^{r'-1}\cB_{r-r',j,\lambda} \,,
\]
so for any fixed $b=(x_0,x_1)$,
\begin{multline*}
\left\|\frac{1}{n^2}\sum_{r=1}^n\sum_{r'=1}^r \cB_{r,\lambda}\cB_{r',\lambda}^{\ast}(b,\cdot)\right\|_{p/2} \\
\le \frac{1}{n^2}\sum_{r=1}^n\sum_{r'=1}^r \|\cB_{r-r',\lambda}(b,\cdot)\|_{p/2} + \frac{1}{n^2}\sum_{r=1}^n\sum_{r'=1}^r\sum_{j=0}^{r'-1}\|\cB_{r-r',j,\lambda}(b,\cdot)\|_{p/2} \,.
\end{multline*}

To calculate these norms, we need to study the kernels of the operators. If $m\le \rho_G$, we have $\cB_{m,\lambda}(b,b') = \frac{1}{|\Im\zeta_{x_0}^{\lambda}(x_1)|^{1/2}}\zeta_{x_0}^{\lambda}(x_1)\cdots\zeta_{x_{m-1}}^{\lambda}(x_m)|\Im\zeta_{x_m}^{\lambda}(x_{m+1})|^{1/2}$ if $b'\in \cB^m b$, otherwise $\cB_{m,\lambda}(b,b')=0$. Here $(x_2;x_{m+1})$ is the path from $(x_0,x_1)$ with $b'=(x_m,x_{m+1})$. On the other hand, several paths outgoing from $b$ can in principle reach $b'\in \cB^m b$, if $m>\rho_G$. As long as $m \le \ell_G$ however, there is at most one additional path of the same length $(x_2';x_{m+1}')$ that can reach $b'$, see Remark~\ref{rem:tangle}. Hence, if $p>4$ and $m \le \ell_G$,
\begin{multline*}
\|\cB_{m,\lambda}(b,\cdot)\|_{p/2} = \left(\sum_{b'\in \cB^m b}|\cB_{m,\lambda}(b,b')|^{p/2}\right)^{2/p} \\
\le \left(2\sum_{(x_2;x_{m+1})}\frac{|\zeta_{x_0}^{\lambda}(x_1)\cdots\zeta_{x_{m-1}}^{\lambda}(x_m)|^{p/2}|\Im\zeta_{x_m}^{\lambda}(x_{m+1})|^{p/4}}{|\Im\zeta_{x_0}^{\lambda}(x_1)|^{p/4}}\right)^{2/p} \le 2^{2/p}(Z_{p/4,\lambda}^{2/p})^m \,.
\end{multline*}
Similarly, we have $\cB_{m,j,\lambda}(b,b') = \frac{\zeta_{x_0}^{\lambda}(x_1)\cdots\zeta_{x_{j+m}}^{\lambda}(x_{j+m+1})\cdot\overline{\zeta_{x_{2j+m+2}}^{\lambda}(x_{2j+m+1})\cdots\zeta_{x_{j+m+2}}^{\lambda}(x_{j+m+1})}}{|\Im\zeta_{x_0}^{\lambda}(x_1)\Im\zeta_{x_{2j+m+1}}^{\lambda}(x_{2j+m})|^{1/2}}\cdot (\frac{|\Im \zeta_{x_{j+m}}^{\lambda}(x_{j+m+1})|}{|\zeta_{x_{j+m}}^{\lambda}(x_{j+m+1})|^2} - |\Im\zeta_{x_{j+m+1}}^{\lambda}(x_{j+m+2})|)$ if $b'\in \cB^{2j+m+1}$ and zero otherwise, if $m+2j+1\le \rho_G$. We bound the term in parentheses as well as the denominator by $2z_{\lambda}^{-2}$ using \eqref{e:zboun2}. This gives
\begin{align*}
\|\cB_{m,j,\lambda}(b,\cdot)\|_{p/2} &\le 2z_{\lambda}^{-2}  \bigg(2\sum_{(x_2;x_{2j+m+2})} |\zeta_{x_0}^{\lambda}(x_1)\cdots\zeta_{x_{j+m}}^{\lambda}(x_{j+m+1}) \\
&\qquad\qquad\qquad\qquad\qquad \cdot \zeta_{x_{j+m+2}}^{\lambda}(x_{j+m+1})\cdots\zeta_{x_{2j+m+1}}^{\lambda}(x_{2j+m})|^{p/2} \bigg)^{2/p} \\
&\le 2^{2/p+1}z_{\lambda}^{-2}z_{\lambda}^{-4} (Z_{p/4,\lambda}^{2/p})^{j}(Z_{p/4,\lambda}^{2/p})^{j+m+1} \le 4 z_{\lambda}^{-6}(Z_{p/4,\lambda}^{2/p})^{2j+m+1} \,,
\end{align*}
where we first used \eqref{e:greensno2} then \eqref{e:greensno}. Denoting $\alpha_{p,\lambda} = Z_{p/4,\lambda}^{2/p}<1$, we get
\[
\left\|\frac{1}{n^2}\sum_{r=1}^n\sum_{r'=1}^r \cB_{r,\lambda}\cB_{r',\lambda}^{\ast}(b,\cdot)\right\|_{p/2} \le \frac{2^{2/p}}{n^2} \sum_{r=1}^n \sum_{r'=1}^r \alpha_{p,\lambda}^{r-r'} + \frac{4z_{\lambda}^{-6}}{n^2} \sum_{r=1}^n\sum_{r'=1}^r\sum_{j=0}^{r'-1}\alpha_{p,\lambda}^{2j+r-r'+1} \le \frac{C_{p,\lambda}}{n} \,,
\]
since $\sum_{r'=1}^r\alpha^{-r'} = \alpha^{-1} \frac{\alpha^{-r}-1}{\alpha^{-1}-1} \le \frac{\alpha^{-r}}{1-\alpha}$ and $\sum_{j=0}^{r'-1}\alpha^{2j} \le \frac{1}{1-\alpha^2}$. To take an explicit $C_{p,\lambda}$, the above is bounded more precisely by $\frac{1}{n} \frac{1}{1-\alpha}(2^{2/p}+\frac{4z_{\lambda}^{-6}}{1-\alpha^2})$. Now $2^{2/p}\le 2\le \frac{2z_{\lambda}^{-6}}{1-\alpha^2}$ since $z_{\lambda}\le 1$. Also, $\frac{1}{1-\alpha^2}\le \frac{1}{1-\alpha}$. So we may take $C_{p,\lambda} = \frac{6z_{\lambda}^{-6}}{(1-\alpha)^2}$.

Similarly, $\| \frac{1}{n^2}\sum_{r=1}^n\sum_{r'=r+1}^n \cB_{r,\lambda}\cB_{r',\lambda}^{\ast}(b,\cdot)\|_{p/2} \le \frac{C_{p,\lambda}}{n}$. Finally, as $\cM_{n,\lambda}\cM_{n,\lambda}^{\ast}$ is self-adjoint, we deduce from Young's inequality that its $p' \to p$ norm is bounded by $\frac{2C_{p,\lambda}}{n}$.

Recalling \eqref{e:pnormpsif}--\eqref{e:ttstarn}, since $\|f_{\lambda}\|_2 \le \sqrt{D}(1+z_{\lambda}^{-1})\|\psi_{\lambda}\|_2$, then arguing similarly for the term containing $g_{\lambda}$, it follows that $\|\psi_{\lambda}\|_p \le \frac{2z_{\lambda}^{-1}\sqrt{2DC_{p,\lambda}}(1+z_{\lambda}^{-1})}{\sqrt{n}}\|\psi_{\lambda}\|_2$. As this holds for any $n\le \ell_G$, the proof is complete (using $\sqrt{12D}=2\sqrt{3D}\le 2D$).

\begin{rem}
The previous proof continues to hold for $n>\ell_G$, as long as for any $\beta>0$, there are at most $2^{\beta k}$ paths of length $k$ between two edges, for $k\le n$. In this case, in the previous argument we obtain $(2^{\beta} Z_{p/4,\lambda}^{2/p})^k$ instead of $(Z_{p/4,\lambda}^{2/p})^k$, and we choose $\beta$ small enough so that $2^{\beta} Z_{p/4,\lambda}^{2/p}<1$.
\end{rem}

\subsection{Support of eigenfunctions}

We first estimate the support of non-backtracking eigenfunctions $f_{\lambda}$ from \eqref{e:fglambda}. Let $S$ be the support of $f_{\lambda}$.

Recall that we have $(\zeta^{\lambda}\cB)^r f_{\lambda} = f_{\lambda}$. Now
\[
\left\|f_{\lambda}\right\|_2^2 = \langle f_{\lambda} \chi_S, f_{\lambda} \rangle = \langle f_{\lambda}\chi_S, (\zeta^{\lambda}\cB)^r f_{\lambda} \rangle \le \|f_{\lambda}\chi_S\|_1\cdot  \| (\zeta^{\lambda}\cB)^r f_{\lambda} \|_{\infty} \,,
\]
and assuming $r\le \ell_G$,
\begin{align*}
\left|[(\zeta^{\lambda}\cB)^rf_{\lambda}](x_0,x_1)\right| & = \left|\sum_{(x_2;x_{r+1})} \zeta_{x_0}^{\lambda}(x_1)\cdots\zeta_{x_{r-1}}^{\lambda}(x_r)f_{\lambda}(x_r,x_{r+1})\right| \\
& \le \sup_{(x_0;x_{r+1})} |\zeta_{x_0}^{\lambda}(x_1)\cdots\zeta_{x_{r-1}}^{\lambda}(x_r)| \sum_{(x_2;x_{r+1})} |f_{\lambda}(x_r,x_{r+1})| \\
& \le 2z_{\lambda}^{-1} (Z_{s,\lambda}^{1/2s})^r \|f_{\lambda}\|_1 = 2z_{\lambda}^{-1} (Z_{s,\lambda}^{1/2s})^r \|f_{\lambda}\chi_S\|_1 \,,
\end{align*}
where we used \eqref{e:inftyboun} and the fact that there are at most two paths of length $r$ between two edges if $r\le \ell_G$. Thus,
\[
\| f_{\lambda}\|_2^2 \le 2z_{\lambda}^{-1}(Z_{s,\lambda}^{1/2s})^r \|f_{\lambda} \chi_S\|_1^2 \le 2z_{\lambda}^{-1}(Z_{s,\lambda}^{1/2s})^r \cdot |S| \cdot \|f_{\lambda}\|_2^2 \,,
\]
using the Cauchy-Schwarz inequality.

It follows that\footnote{in case of $(q+1)$-regular graphs, $W\equiv 0$, this argument yields more precisely $|S| \ge \frac{1}{2} q^{r/2}$.}
\[
|S| \ge \frac{1}{2}z_{\lambda}(Z_{s,\lambda}^{-1/2s})^r \,,
\]
with $Z_{s,\lambda}^{-1/2s} >1$.

To go back to the support of $\psi_{\lambda}$, recall that $f_{\lambda}(x_0,x_1) = \psi_{\lambda}(x_1) - \zeta_{x_0}^{\lambda}(x_1)\psi_{\lambda}(x_0)$. If $f_{\lambda}(x_0,x_1)\neq 0$, then either $\psi_{\lambda}(x_0) \neq 0$ or $\psi_{\lambda}(x_1)\neq 0$. Hence, if $\Lambda=\supp \psi_{\lambda}$ and
\[
\widetilde{S} = \{(x,y),(y,x):x\in \Lambda,y\sim x\} \,,
\]
we have $S\subseteq \widetilde{S}$. In particular, $|S| \le |\widetilde{S}| \le 2D \cdot |\Lambda|$, so
\[
|\Lambda| \ge \frac{z_{\lambda}}{4D} M_{\lambda}^r
\]
with $M_{\lambda}=Z_{2,\lambda}^{-1/4}>1$.

To obtain the sharper lower bound (without $z_{\lambda}$), just replace the function $f_{\lambda}$ in the previous argument by $\frac{f_{\lambda}}{|\Im \zeta^{\lambda}|^{1/2}}$, and the operator $(\zeta^{\lambda}\cB)^r$ by $\cB_{r,\lambda} = \frac{1}{|\Im \zeta^{\lambda}|^{1/2}}(\zeta^{\lambda}\cB)^r|\Im\zeta^{\lambda}|^{1/2}$. In this case, we get to estimate $\sup_{(x_0;x_{r+1})} \frac{|\zeta_{x_0}^{\lambda}(x_1)\cdots\zeta_{x_{r-1}}^{\lambda}(x_r)|\cdot|\Im\zeta_{x_r}^{\lambda}(x_{r+1})|^{1/2}}{|\Im \zeta_{x_0}^{\lambda}(x_1)|^{1/2}}$, and from the argument preceding \eqref{e:inftyboun}, we see this is bounded by $Z_{s,\lambda}^{r/2s}$ (instead of $z_{\lambda}^{-1}Z_{s,\lambda}^{r/2s}$ above). 

\bigskip

{\bf{Acknowledgements~:}} E.L.M. was supported by the Marie Sk{\l}odowska-Curie Individual Fellowship grant 703162. M.S. was supported by a public grant as part of the Investissement d'avenir project, reference ANR-11-LABX-0056-LMH, LabEx LMH.

\providecommand{\bysame}{\leavevmode\hbox to3em{\hrulefill}\thinspace}
\providecommand{\MR}{\relax\ifhmode\unskip\space\fi MR }
\providecommand{\MRhref}[2]{%
  \href{http://www.ams.org/mathscinet-getitem?mr=#1}{#2}
}
\providecommand{\href}[2]{#2}


\begin{thebibliography}{10} 

\bibitem{AW15}
M.~Aizenman, S.~Warzel, \emph{Random Operators. Disorder Effects on Quantum Spectra and Dynamics}, GSM 168, AMS 2015.

\bibitem{Altshuler1}
A.~De~Luca, B.~L.~Altshuler, V.~E.~Kravtsov, and A.~Scardicchio, \emph{Anderson Localization on the Bethe Lattice: Nonergodicity of Extended States}, Phys. Rev. Lett. \textbf{113} (2014) 046806.

\bibitem{Altshuler2}
A.~De~Luca, B.~L.~Altshuler, V.~E.~Kravtsov, and A.~Scardicchio, \emph{Support set of random wave-functions on the Bethe lattice}, arXiv 2013.

%
\bibitem{AL}
D.~Aldous, R.~Lyons, \emph{Processes on unimodular random networks}, Electron. J. Probab. \textbf{12} (2007) 1454--1508.
%
\bibitem{AL02}
A.~Amit, N.~Linial, \emph{Random graph coverings I: General theory and graph connectivity}, Combinatorica \textbf{22} (2002) 1--18.

\bibitem{ALM02}
A.~Amit, N.~Linial, J.~Matou\v{s}ek, \emph{Random lifts of graphs: independence and chromatic number}, Random Structures \& Algorithms \textbf{20} (2002) 1--22.

\bibitem{A}
N.~Anantharaman, \emph{Quantum ergodicity on regular graphs}, Comm. Math. Phys. \textbf{353} (2017) 633--690.


\bibitem{ALM}
N.~Anantharaman, E.~Le~Masson, \emph{Quantum ergodicity on large regular graphs}, Duke Math. Jour. \textbf{164} (2015) 723--765.

\bibitem{AS2}
N.~Anantharaman, M.~Sabri, \emph{Quantum ergodicity on graphs : from spectral to spatial delocalization}, arXiv:1704.02766. To appear in the Annals of Math.

\bibitem{AS3}
N.~Anantharaman, M.~Sabri, \emph{Quantum ergodicity for the Anderson model on regular graphs},  J. Math. Phys. \textbf{58} (2017), 091901.
%
\bibitem{AS4}
N.~Anantharaman, M.~Sabri, \emph{Recent results of quantum ergodicity on graphs and further investigation}, arXiv:1711.07666. To appear in Annales de la facult\'e des Sciences de Toulouse.
%
\bibitem{AS}
N.~Anantharaman, M.~Sabri, \emph{Poisson kernel expansions for Schr\"odinger operators on trees}, J. Spectr. Theory \textbf{9} (2019), 243--268.

%
\bibitem{BS}
I.~Benjamini, O.~Schramm, \emph{Recurrence of distributional limits of finite planar graphs}, Electron. J. Probab. \textbf{6} (2001) 13 pp.

\bibitem{BHY}
R.~Bauerschmidt, J.~Huang, H-T.~Yau, \emph{Local Kesten-McKay law for random regular graphs}, Comm. Math. Phys. Published online - arXiv:1609.09052. 

\bibitem{BL06}
Y.~Bilu, N.~Linial, \emph{Lifts, discrepancy and nearly optimal spectral gap}, Combinatorica \textbf{26} (2006) 495--519.

\bibitem{Bornew}
C.~Bordenave, \emph{A new proof of Friedman's second eigenvalue Theorem and its extension to random lifts}, To appear in Annales de l'ENS - arXiv:1502.04482.

\bibitem{Bor}
C.~Bordenave, \emph{Spectral measures of random graphs}, lecture notes, author's homepage.

\bibitem{BSV17}
C.~Bordenave, A.~Sen, B.~Virag, \emph{Mean quantum percolation},  J. Eur. Math. Soc. (JEMS) \textbf{19} (2017) 3679--3707.

\bibitem{BDGHT}
G. Brito, I. Dumitriu, S. Ganguly, C. Hoffman, L. V. Tran, \emph{Recovery and Rigidity in a Regular Stochastic Block Model}, arXiv:1507.00930.

\bibitem{BL}
S.~Brooks, E.~Lindenstrauss, \emph{Non-localization of eigenfunctions on large regular graphs}, Israel J. Math. \textbf{193} (2013) 1--14.

\bibitem{BL17}
S.~Brooks, E.~Le~Masson, \emph{$L^p$ norms of eigenfunctions on regular graphs and on the sphere}, Int. Math. Res. Not. published online.

\bibitem{BLL}
S.~Brooks, E.~Le~Masson, E.~Lindenstrauss \emph{Quantum ergodicity and averaging operators on the sphere},  Int. Math. Res. Not. \textbf{19} (2016) 6034--6064.
%
%
%
%
%
%
%
%
%


\bibitem{Fri03}
J.~Friedman, \emph{Relative expanders or weakly relatively Ramanujan graphs}, Duke Math. J. \textbf{118} (2003), 19--35.

\bibitem{Gei}
L~Geisinger, \emph{Convergence of the density of states and delocalization of eigenvectors on random regular graphs}, J. Spectr. Theory \textbf{5} (2015) 783--827.



\bibitem{HT15}
A.~Hassell, M.~Tacy, \emph{Improvement of eigenfunction estimates on manifolds of nonpositive curvature}, Forum Math. \textbf{27} (2015) 1435--1451.

\bibitem{KS}
T.~Kawarabayashi, M.~Suzuki, \emph{Decay rate of the Green function in a random potential on the
Bethe lattice and a criterion for localization}, J. Phys. A. Math. Gen. \textbf{26} (1993) 5729--5750.

\bibitem{KLW2}
M.~Keller, D.~Lenz, S.~Warzel, \emph{On the spectral theory of trees with finite cone type}, Israel J. Math. \textbf{194} (2013) 107--135.

%

\bibitem{Ki08}
W.~Kirsch, \emph{An invitation to random {S}chr\"odinger operators}, in Random {S}chr\"odinger operators, Panor. Synth\`eses \textbf{25} (2008) 1--119.

\bibitem{Klein}
A.~Klein, \emph{Extended states in the Anderson model on the Bethe lattice}, Adv. Math. \textbf{133} (1998) 163--184.

%

\bibitem{MWW04}
B.~D.~McKay, N.~C.~Wormald, B.~Wysocka, \emph{Short cycles in random regular graphs}, Electron. J. Combin. \textbf{11} (2004) 12 pp.

\bibitem{MC}
F.~L.~Metz and I.~P.~Castillo, \emph{Level compressibility for the Anderson model on regular random graphs and the eigenvalue statistics in the extended phase}, Phys. Rev. B \textbf{96} (2017) 064202.

\bibitem{OW07}
R.~Ortner and W.~Woess, \emph{Non-backtracking random walks and cogrowth of graphs}, Canad. J. Math., \textbf{59} (2007) 828--844.


\bibitem{Pu15}
D.~Puder, \emph{Expansion of Random Graphs: New Proofs, New Results}, Inventiones Mathematicae \textbf{201} (2015) 845--908


\bibitem{Si95}
B.~Simon, \emph{Spectral  analysis of  rank one  perturbations  and  applications}, in Mathematical Quantum Theory. II. Schr\"odinger Operators, AMS 1995

\bibitem{Si11}
B.~Simon, \emph{Szeg\H o's theorem and its descendants. Spectral theory for $L^2$ perturbations of orthogonal polynomials}, M. B. Porter Lectures, PUP 2011.
%

\bibitem{TMS}
K.~S.~Tikhonov, A.~D.~Mirlin, and M.~A.~Skvortsov, \emph{Anderson localization and ergodicity on random regular graphs}, Phys. Rev. B, \textbf{94} (2016) 220203.

\end{thebibliography}
\end{document}